\documentclass[12pt]{amsart}
\usepackage{amssymb,latexsym,amsmath,amscd,amsthm,graphicx, color}
\usepackage[all]{xy}
\usepackage{pgf,tikz}
\usepackage{mathrsfs}
\usetikzlibrary{arrows}
\definecolor{uuuuuu}{rgb}{0.26666666666666666,0.26666666666666666,0.26666666666666666}
\definecolor{xdxdff}{rgb}{0.49019607843137253,0.49019607843137253,1.}
\definecolor{ffqqqq}{rgb}{1.,0.,0.}

\definecolor{uuuuuu}{rgb}{0.26666666666666666,0.26666666666666666,0.26666666666666666}
\definecolor{qqwuqq}{rgb}{0.,0.39215686274509803,0.}
\definecolor{zzttqq}{rgb}{0.6,0.2,0.}
\definecolor{xdxdff}{rgb}{0.49019607843137253,0.49019607843137253,1.}
\definecolor{qqqqff}{rgb}{0.,0.,1.}
\definecolor{cqcqcq}{rgb}{0.7529411764705882,0.7529411764705882,0.7529411764705882}

\theoremstyle{plain}

\newtheorem{lemma}[subsection]{Lemma}

\newtheorem{defi}[subsection]{Definition}
\theoremstyle{definition}
\newtheorem{prop}[subsection]{Proposition}
\newtheorem{cor}[subsection]{Corollary}
\newtheorem{example}[subsection]{Example}

\newtheorem{remark}[subsection]{Remark}

\newtheorem{note}[subsection]{Note}

\newcommand{\uu}{\cup}% union
\newcommand{\ii}{\cap}% intersection
\newcommand{\sci}{\subset}% strictly contained in
\newcommand{\es}{\emptyset}% the empty set
\newcommand{\set}[1]{\{#1\}}% set
\newcommand{\ga}{\alpha}
\newcommand{\gb}{\beta}

\newcommand{\gd}{\delta}
\renewcommand{\gg}{\gamma}% old use >>

\newcommand{\go}{\omega}

\newcommand{\gs}{\sigma}
\newcommand{\gt}{\tau}

\newcommand{\gO}{\Omega}

\newcommand{\tbf}{\textbf}% text bold
\newcommand{\tit}{\textit}% text italic

\newcommand{\D}[1]{\mathbb{#1}}% Doubled - blackboard bold - only caps, uas as \D{A}
\newcommand{\F}[1]{\mathfrak{#1}}% Fraktur, use as \F{a}

\newcommand{\te}{\text}% same as \mathrm command.

\begin{document}

%\title{Probability distributions supported on sets generated by infinite affine transformations and optimal quantization}
To appear, Fractal and Fractional
\title{Quantization for infinite affine transformations}

\author{Do\u gan \c C\"omez}
\address{Department of Mathematics \\
408E24 Minard Hall,
North Dakota State University
\\
Fargo, ND 58108-6050, USA.}
\email{Dogan.Comez@ndsu.edu}

\author{ Mrinal Kanti Roychowdhury}
\address{School of Mathematical and Statistical Sciences\\
University of Texas Rio Grande Valley\\
1201 West University Drive\\
Edinburg, TX 78539-2999, USA.}
\email{mrinal.roychowdhury@utrgv.edu/roychowdhurymk@gmail.com}

\subjclass[2010]{60Exx, 28A80, 94A34.}
\keywords{Affine transformations, affine set, affine measure, optimal quantizers, quantization error}
% \thanks{The research of the second author was supported by U.S. National Security Agency (NSA) Grant H98230-14-1-0320}

\date{}
\maketitle

\pagestyle{myheadings}\markboth{Do\u gan \c C\"omez and Mrinal Kanti Roychowdhury}{Quantization for infinite affine transformations}

\begin{abstract} Quantization for a probability distribution refers to the idea of estimating a given probability by a discrete probability supported by a finite set.  In this article, we consider a probability distribution generated by an infinite system of affine transformations $\{S_{ij}\}$ on $\mathbb R^2$ with associated probabilities $\{p_{ij}\}$ such that $p_{ij}>0$ for all $i, j\in \mathbb N$ and $\sum_{i, j=1}^\infty p_{ij}=1$. For such a probability measure $P$, the optimal sets of $n$-means and the $n$th quantization error are calculated for every natural number $n$. It is shown that the distribution of such a probability measure is the same as that of the direct product of the Cantor distribution.  In addition, it is proved that the quantization dimension $D(P)$ exists and is finite; whereas, the $D(P)$-dimensional quantization coefficient does not exist, and the $D(P)$-dimensional lower and the upper quantization coefficients lie in the closed interval $[\frac{1}{12}, \frac{5}{4}]$.
\end{abstract}

\section{Introduction}
The quantization problem for probability measures is concerned with approximating a given measure by discrete measures of finite support in $L_r$-metrics. This problem has root in information theory and engineering technology, in particular in signal processing and pattern recognition (\cite{BW, GN}).
For a Borel probability measure $P$ on $\D R^d,$ a quantizer is a function $q$ mapping $d$-dimensional vectors in the domain $\gO \sci \D R^d$ into a finite set of vectors $\ga\sci \D R^d$.
In this case, the error $\int \min_{a \in \ga} \|x-a\|^2 dP(x),$ where $\|\cdot\|$ is the Euclidean norm $\D R^d, $ is often referred to as the \tit{variance, cost,} or \tit{distortion error} for $\ga$ with respect to the measure $P$, and is denoted by $V(\ga):= V(P; \ga)$. The value $\inf\set{V(P; \ga) :\alpha \subset \mathbb R^d, \text{ card}(\alpha) \leq n}$ is called the \tit{$n$th quantization error} for the $P$, and is denoted by $V_n:=V_n(P)$.  A set $\ga$ on which this infimum is attained and contains no more than $n$ points is called an \tit{optimal set of $n$-means}. The elements of an optimal set are called \tit{optimal quantizers}.  It is known that for a Borel probability measure $P$ if its support contains infinitely many elements and $\int \| x\|^2 dP(x)$ is finite, then an optimal set of $n$-means always has exactly $n$-elements \cite{AW, GKL, GL1, GL2}.  The number
$\lim_{n\to \infty}  \frac{2\log n}{-\log V_n(P)}, $ if exists, is called the \tit{quantization dimension} of the measure $P$, and is denoted by $D(P)$; likewise, for any $s\in (0, +\infty)$, the number $\lim\limits_{n\to\infty} n^{\frac 2 s} V_n(P)$, if exists, is called the \tit{$s$-dimensional quantization coefficient} for $P$.

For a finite set $\ga \subset \D R^d,$ the Voronoi region generated by $a\in \ga,$ denoted by $M(a|\ga),$ is the set of all points in $\D R^d$ which are closer to $a \in \ga$ than to all other elements in $\ga . $
%the set $\set{M(a|\ga) : a \in \ga}$ is called the \tit{Voronoi diagram} or \tit{Voronoi tessellation} of $\D R^d$.
%A special quantization scheme is given by the Voronoi tessellation which associates with each $a\in \ga$ its Voronoi region $M(a|\ga)$.
For a probability distribution $P$ on $\D R^d$ the centroids of the regions $M(a|\ga)$ are given by $a^\ast=\frac{1}{P(M (a|\ga))}\int_{M(a|\ga)} x dP. $
%\begin{align*}
%a^\ast=\frac{1}{P(M (a|\ga))}\int_{M(a|\ga)} x dP
%= \frac{\int_{M(a|\ga)} x dP}{\int_{M(a|\ga)} dP}.
%\end{align*}
A Voronoi tessellation is called a \tit{centroidal Voronoi tessellation} (CVT) if $a^\ast=a$, i.e., if the generators are also the centroids of their own Voronoi regions.
For a Borel probability measure $P$ on $\D R^d$, an optimal set of $n$-means forms a CVT; however, the converse is not true in general \cite{DFG, R3}.
%The number
%\[\lim_{n\to \infty}  \frac{2\log n}{-\log V_n(P)},\] if it exists, is called the \tit{quantization dimension} of the probability measure $P$, and is denoted by $D(P)$; likewise, for any $s\in (0, +\infty)$, the number $\lim\limits_{n\to\infty} n^{\frac 2 s} V_n(P)$, if it exists, is called the $s$-dimensional \tit{quantization coefficient} for $P$ (\cite{GL1, P}).
%A Borel measurable partition $\set{A_a : a\in \ga}$ of $\D R^d$ is called a Voronoi partition of $\D R^d$ if $A_a\sci M(a|\ga)$ for every $a\in \ga$.
The following fact is known \cite{GG, GL2}:
\begin{prop} \label{prop000}
Let $\ga$ be an optimal set of $n$-means and $a\in \ga$. Then,
\begin{itemize}
\item[(i)] $P(M(a|\ga))>0$ and $ P(\partial M(a|\ga))=0$,
\item[(ii)] $a=E(X : X \in M(a|\ga))$, where $X$ is a random variable with distribution $P,$
\item[(iii)] $P$-almost surely the set $\set{M(a|\ga) : a \in \ga}$ forms a Voronoi partition of $\D R^d$.
\end{itemize}
%where $X$ is a random variable with distribution $P.$
%\item[(iii)] $P$-almost surely the set $\set{M(a|\ga) : a \in \ga}$ forms a Voronoi partition of $\D R^d$.
%\end{itemize}
\end{prop}

%A transformation $S: X \to X$ on a metric space $(X, d)$ is called a \tit{contractive} or \tit{contraction mapping} if there is a constant $c\in (0,1)$ such that $d(S(x), S(y))\leq c d(x, y)$ for all $x, y \in X$.  $S$ is called \tit{similarity mapping} or \tit{similitude} if $d(S(x), S(y))=s d(x, y),\ \forall \ x, y\in X,$ for some $s>0 $. Here $s$ is called the \tit{similarity ratio} or the \tit{similarity constant} of $S$.

Let $X=\mathbb{R}$ and consider the probability distribution $P_c:=\frac 12 P_c\circ U_1^{-1}+\frac 12 P_c\circ U_2^{-1},$ where $U_1(x)=\frac 13 x$ and $U_2(x)=\frac 13 x+\frac 23 , $ for all $x\in \D R$.  Since its support is the standard Cantor set generated by $U_1$ and $U_2$, $P_c$ is called the Cantor distribution.  S. Graf and H. Luschgy determined the optimal sets of $n$-means and the $n$th quantization errors  for the Cantor distribution, for all $n \geq 1,$ completing its quantization program \cite{GL3}.  This result has been extended to the setting of a nonuniform Cantor distribution by L. Roychowdhury \cite{R1}.  Analogously, the Cantor dust is generated by the contractive mappings $\{S_i \}_{i=1}^4 $ on $\D R^2,$ where $S_1(x_1, x_2)=\frac 13(x_1, x_2)$, $S_2(x_1, x_2)=\frac 13(x_1, x_2) + (\frac 23, 0)$, $S_3(x_1, x_2)=\frac 13(x_1, x_2) +(0, \frac 23)$, and $S_4(x_1, x_2)=\frac 13(x_1, x_2)+(\frac 23, \frac 23)$. If $P$ is a Borel probability measure on $\D R^2$ such that $P=\frac 1 4P\circ S_1^{-1}+\frac 1 4P\circ S_2^{-1}+\frac 1 4P\circ S_3^{-1}+\frac 1 4P\circ S_4^{-1}$, then $P$ has support the Cantor dust. For this measure, D. \c C\"omez and M.K. Roychowdhury determined the optimal sets of $n$-means and the $n$th quantization errors \cite{CR}. Let $P$ be a probability measure on $\D R$ generated by an infinite collection of similitudes $\set{S_j}_{j=1}^\infty , $ where $S_j(x)=\frac 1 {3^j} x+1-\frac{1}{3^{j-1}}$ for all $x\in\D R$ and $P$ is given by $P=\sum_{j=1}^\infty \frac 1{2^j} P\circ S_j^{-1}$. For this measure, M.K. Roychowdhury determined the optimal sets of $n$-means and the $n$th quantization errors \cite{R2}, which is an infinite extension of the result of S. Graf and H. Luschgy in \cite{GL3}.  The quantization dimension for probability distributions generated by an infinite collection of similitudes was determined by E. Mihailescu and M.K. Roychowdhury in \cite{MR}, which is an infinite extension of the result of S. Graf and H. Luschgy in \cite{GL4}.
In this article, we  study extension of the result of D. \c C\"omez and M.K. Roychowdhury in \cite{CR} to the setting of countably infinite affine maps on $\D R^2$, which will also complete the program initiated in \cite{MR}.

Let $\set{S_{(i, j)} : i, j\in \D N}$ be a collection of countably infinite affine transformations on $\D R^2$, where $S_{(i, j)}(x_1, x_2)=(r^i x_1+1-r^{i-1}, r^j x_2+1-r^{j-1})$, where $0<r \leq \frac{1}{3}. $  Clearly, these affine transformations are all contractive but are not similarity mappings. Associate the mappings $S_{(i, j)}$ with the probabilities $p_{(i,j)}$ such that $p_{(i, j)}=\frac 1{2^{i+j}}$ for all $i, j\in \D N$, where $\D N:=\set{1, 2, 3, \cdots}$. Then, there exists a unique Borel probability measure $P$ on $\D R^2$  (\cite{H}, \cite{MaU}, \cite{M}, etc.) such that
\[P=\sum_{i, j=1}^\infty p_{(i, j)} P\circ S_{(i, j)}^{-1}.\]
The support of such a probability measure lies in the unit square $[0,1]^2 .$  We call such a measure an \tit{affine measure} on $\D R^2, $ or more specifically, an \tit{infinitely generated affine measure} on $\D R^2$. This article deals with the quantization of this measure $P$. The arrangement of the paper is as follows: in Section~2, we discuss the basic definitions and lemmas about the optimal sets of $n$-means and the $n$th quantization errors.  The arguments in this section point out that determining optimal sets of n-means and the $n$th quantization errors for all $n \geq 3$ and for arbitrary $r \in (0, \frac{1}{3})$ require very intricate and complicated analysis; hence, for clarity purposes,
in the remaining sections the focus will be on the case $r=\frac{1}{3} . $  Section 3 is devoted to determining the optimal sets of $n$-means for $n=2$ and $n=3$. In Section~4, we first define a mapping $F$ which enables us to convert the infinitely generated affine measure $P$ to a finitely generated product measure $P_c\times P_c$, each $P_c$ is the Cantor distribution. %given by $P_c=\frac 12 P\circ U_1^{-1}+\frac 12 P_c\circ U_2^{-1}$, where $U_1(x)=\frac{1}{3} x$ and $U_2(x)= \frac{1}{3} x+ \frac{2}{3}$ for all $x\in \D R$.
Having this connection between $P$ and $P_c;$ together with
the optimal sets of $n$-means for $n = 1, 2, 3,$ in Section 5 we will utilize the dynamics of the affine maps to obtain the main results of the paper: closed formulas to determine the optimal sets of $n$-means and the corresponding quantization errors for all $n\geq 4$.  For clarity of the exposition, we also provide some examples and figures to illustrate the constructions.  Lastly, having closed form for the quantization errors for each $n, $ we prove the existence of the quantization dimension $D(P)$ and show that the $D(P)$-dimensional quantization coefficient for $P$ does not exist (but are finite) and the $D(P)$-dimensional lower and the upper quantization coefficients lie in the closed interval $[\frac{1}{12}, \frac{5}{4}]$.

The results and the arguments in this article are not straightforward generalizations of those in \cite{R2}; in particular, this is the case for optimal sets.   By the nature of the affine transformations considered in this paper, the optimal sets of order $n=k^2 ,\ k\geq 1, $ are the same as the cross product of optimal sets of order $k$ obtained in \cite{R2}; however, the same cannot be said for other $n\geq 3.$  Clearly, for $n$ a prime number, optimal sets of $n$-means cannot be obtained this way.  Furthermore, as will be seen from the main theorem, even for $n=kl, $ optimal sets of $n$-means are not the same as the cross product of optimal sets of $k$- and $l$-means in \cite{R2}.  For example, optimal sets of 2- and 3-means in \cite{R2} are $\{\frac16, \frac{5}{6}\} $ and $\{\frac{1}{6}, \frac{13}{18}, \frac{17}{18}\} $ (or $\{\frac{1}{18}, \frac{5}{18}, \frac{5}{6}\}),$ respectively; hence, cross product of these sets produce some of the optimal sets of 6-means.  On the other hand, one of the optimal sets of 6-means is $\{(\frac{1}{18}, \frac{1}{6}), (\frac{5}{6}, \frac{1}{6}), (\frac{13}{18}, \frac{1}{6}), (\frac{1}{18}, \frac{5}{6}), (\frac{13}{18}, \frac{5}{6}), (\frac{5}{6},\frac{5}{6}) \}$, which cannot be obtained as the cross product of optimal sets of 2- and 3-means in \cite{R2}.

\section{Preliminaries}

Let $P$ be the affine measure on $\D R^2$ generated by the affine maps $\set{S_{(i, j)} : i, j\in \D N}$ defined above.
Consider the alphabet $\mathcal{I}=\D N^2=\set{(i, j) : i, j\in \D N}$.  By a \textit{string} or a \textit{word} $\go$ over $\mathcal{I}$, it is meant a finite sequence $\go:=\go_1\go_2\cdots \go_k$
of symbols from the alphabet, $k\geq 1$, where $k$ is called the length of the word $\go$.  A word of length zero is called the \textit{empty word}, and is denoted by $\emptyset$.  By $\mathcal{I}^*$ we denote the set of all words
over the alphabet $\mathcal{I}$ of some finite length $k, $ including the empty word $\emptyset$. By $|\go|$, we denote the length of a word $\go \in \mathcal{I}^\ast$.  For any two words $\go:=\go_1\go_2\cdots \go_k$ and
$\tau:=\tau_1\tau_2\cdots \tau_\ell$ in $\mathcal{I}^\ast$, by
$\go\tau:=\go_1\cdots \go_k\tau_1\cdots \tau_\ell$ we mean the word obtained from the
concatenation of $\go$ and $\tau$. For $n\geq 1$ and $\go=\go_1\go_2\cdots\go_n\in \mathcal{I}^\ast$ we define $\go^-:=\go_1\go_2\cdots\go_{n-1}$. Note that $\go^-$ is the empty word if the length of $\go$ is one. Analogously, by $\D N^\ast$ we denote the set of all words over the alphabet $\D N$, and for any $\gt\in \D N^\ast , \  |\gt|$, $\gt^-$, etc. are defined similarly.  Let $\go \in \mathcal{I}^k$, $k\geq 1$, be such that $\go=(i_1, j_1)(i_2, j_2)\cdots(i_k, j_k)$, then $\go^{(1)}$ and $\go^{(2)}$ will denote the ``coordinate words"; i.e., $\go^{(1)}:=i_1i_2\cdots i_k$ and $\go^{(2)}:=j_1j_2\cdots j_k$. Thus, $\go_{|\go|}^{(1)}=i_k$ and $\go_{|\go|}^{(2)}=j_k$. These lead us to define the following notations: For $\go \in \mathcal{I}^\ast$, by $\go (\es, \infty)$ it is meant the set of all words $\go^-(\go^{(1)}_{|\go|}, \go^{(2)}_{|\go|}+j)$ obtained by concatenating the word $\go^-$ with the word $(\go^{(1)}_{|\go|}, \go^{(2)}_{|\go|}+j)$ for $j\in \D N$, i.e.,
\[\go (\es, \infty):=\set{\go^-(\go^{(1)}_{|\go|}, \go^{(2)}_{|\go|}+j) : j\in \D N}.\]
Similarly,  $\go (\infty, \es)$ and $\go (\infty, \infty)$ represent the sets
\[\go (\infty, \es):=\set{\go^-(\go^{(1)}_{|\go|}+i, \go^{(2)}_{|\go|}) : i\in \D N} \te{ and } \go (\infty, \infty):=\set{\go^-(\go^{(1)}_{|\go|}+i, \go^{(2)}_{|\go|}+j) : i, j\in \D N},\]
respectively. Analogously, for any $\gt \in \D N^\ast$, by $(\gt, \infty)$ it is meant the set $(\gt, \infty):=\set{\gt+i : i \in \D N}$, and $(\gt, \es)$ represents the set $(\gt, \es):=\set{\gt}$.
Thus, if $\go=(i_1, j_1)(i_2, j_2)\cdots(i_k, j_k)(\infty, \es)$, then we write $\go^{(1)}:=(i_1i_2\cdots i_k,\infty)$ and $\go^{(2)}:=j_1j_2\cdots j_k$; if $\go=(i_1, j_1)(i_2, j_2)\cdots(i_k, j_k)(\es, \infty)$, then we write $\go^{(1)}:=i_1i_2\cdots i_k$ and $\go^{(2)}:=(j_1j_2\cdots j_k, \infty)$; and if $\go=(i_1, j_1)(i_2, j_2)\cdots(i_k, j_k)(\infty, \infty)$, then we write $\go^{(1)}:=(i_1i_2\cdots i_k, \infty)$ and $\go^{(2)}:=(j_1j_2\cdots j_k, \infty)$.
For $\go=\go_1 \go_2 \cdots \go_k \in \mathcal{I}^k$,  $k\geq 1,$ let us write
\begin{align*}
S_\go:&=S_{\go_1}\circ \cdots \circ S_{\go_k}, \, p_\go:=p_{\go_1} p_{\go_2} \cdots p_{\go_k} \text{ and } J_\go:=S_\go([0, 1]\times [0, 1]).
\end{align*}
In particular, $S_\es =I, $ the identity mapping on $\D R^2, $
%If $\go$ is the empty word $\es$, by $S_\go$ we mean the identity mapping on $\D R^2$
and $J:=J_\es=S_\es([0, 1]\times[0, 1])$.
Then, the probability measure $P$ has support the closure of the limit set $\mathcal{S} $, where $\mathcal{S}=\bigcap_{k\in \D  N} \bigcup_{\go \in \mathcal{I}^k} J_\go$. The limit set $\mathcal{S}$ is called the \tit{affine set} or \tit{infinitely generated affine set}.  For  $\go \in \mathcal{I}^k$ and $i, j \in \D N$, the rectangles $J_{\go (i, j)}$, into which $J_\go$ is split up at the $(k+1)$th level are called the \tit{children} or the \tit{basic rectangles} of $J_\go$ (see Figure~\ref{Fig1}). For $\go \in \mathcal{I}^\ast$, we write
$$   J_{\go(\es, \infty)}:=\mathop{\uu}\limits_{j=1}^\infty J_{\go^-(\go^{(1)}_{|\go|}, \, \go^{(2)}_{|\go|}+j)}, \ \
 J_{\go(\infty, \es)}:=\mathop{\uu}\limits_{i=1}^\infty J_{\go^-(\go^{(1)}_{|\go|}+i,\, \go^{(2)}_{|\go|})}, \ \
J_{\go(\infty, \infty)}:=\mathop{\uu}\limits_{i, j=1}^\infty J_{\go^-(\go^{(1)}_{|\go|}+i,\,  \go^{(2)}_{|\go|}+j)}; $$
$$ \aligned & p_{\go(\es, \infty)}:=P(J_{\go(\es, \infty)})=\mathop{\sum}\limits_{j=1}^\infty p_{\go^-(\go^{(1)}_{|\go|}, \, \go^{(2)}_{|\go|}+j)}, \ \
 p_{\go(\infty, \es)}:=P(J_{\go(\infty, \es)})=\mathop{\sum}\limits_{i=1}^\infty p_{\go^-(\go^{(1)}_{|\go|}+i,\, \go^{(2)}_{|\go|})}, \ \text{and} \\
& p_{\go(\infty, \infty)}:=P(J_{\go(\infty, \infty)})=\mathop{\sum}\limits_{i, j=1}^\infty p_{\go^-(\go^{(1)}_{|\go|}+i,\,  \go^{(2)}_{|\go|}+j)}.\endaligned $$
Notice that for any $\go \in \mathcal{I}^\ast$,
$p_{\go(\es, \infty)}=p_{\go^-}\sum_{j=1}^\infty \frac{1}{2^{\go^{(1)}_{|\go|}+\go^{(2)}_{|\go|}+j}}=p_{\go^-}p_{\go_{|\go|}}\sum_{j=1}^\infty \frac 1{2^j}=p_{\go^-}p_{\go_{|\go|}}=p_\go$; and similarly,
$p_{\go(\infty, \es)}=p_{\go(\infty, \infty)}=p_\go.$

\begin{figure}
\begin{tikzpicture}[line cap=round,line join=round,>=triangle 45,x=0.5 cm,y=0.5 cm]
\clip(-2.4908520423188394,-0.4928312489533277) rectangle (12.340086210009787,12.653642597251718);
\draw (0.,12.)-- (12.,12.);
\draw (12.,12.)-- (12.,0.);
\draw (12.,0.)-- (0.,0.);
\draw (0.,0.)-- (0.,12.);
\draw (0.,0.)-- (4.,0.);
\draw (0.,4.)-- (4.,4.);
\draw (4.,4.)-- (4.,0.);
\draw (0.,8.)-- (4.,8.);
\draw (0.,9.333333333333332)-- (4.,9.333333333333332);
\draw (4.,9.333333333333332)-- (4.,8.);
\draw (0.,10.666666666666666)-- (0.,11.11111111111111);
\draw (0.,10.666666666666666)-- (4.,10.666666666666666);
\draw (4.,10.666666666666666)-- (4.,11.11111111111111);
\draw (4.,11.11111111111111)-- (0.,11.11111111111111);
\draw (8.,0.)-- (9.333333333333332,0.);
\draw (8.,0.)-- (9.333333333333334,0.);
\draw (9.333333333333334,0.)-- (9.333333333333334,4.);
\draw (9.333333333333334,4.)-- (8.,4.);
\draw (8.,0.)-- (8.,4.);
\draw (8.,8.)-- (8.,9.333333333333334);
\draw (8.,9.333333333333334)-- (9.333333333333334,9.333333333333334);
\draw (9.333333333333334,9.333333333333334)-- (9.333333333333334,8.);
\draw (9.333333333333334,8.)-- (8.,8.);
\draw (8.,10.666666666666666)-- (9.333333333333334,10.666666666666666);
\draw (9.333333333333334,10.666666666666666)-- (9.333333333333334,11.11111111111111);
\draw (9.333333333333334,11.11111111111111)-- (8.,11.11111111111111);
\draw (8.,11.11111111111111)-- (8.,10.666666666666666);
\draw (10.666666666666666,0.)-- (11.11111111111111,0.);
\draw (10.666666666666666,0.)-- (10.666666666666666,4.);
\draw (10.666666666666666,4.)-- (11.11111111111111,4.);
\draw (11.11111111111111,4.)-- (11.11111111111111,0.);
\draw (10.666666666666666,8.)-- (11.11111111111111,8.);
\draw (11.11111111111111,8.)-- (11.11111111111111,9.333333333333334);
\draw (11.11111111111111,9.333333333333334)-- (10.666666666666666,9.333333333333334);
\draw (10.666666666666666,9.333333333333334)-- (10.666666666666666,8.);
\draw (10.666666666666666,10.666666666666666)-- (11.11111111111111,10.666666666666666);
\draw (11.11111111111111,10.666666666666666)-- (11.11111111111111,11.11111111111111);
\draw (11.11111111111111,11.11111111111111)-- (10.666666666666666,11.11111111111111);
\draw (10.666666666666666,11.11111111111111)-- (10.666666666666666,10.666666666666666);
\draw (10.504904953108322,12.508328849994291) node[anchor=north west] {$\vdots$};
\draw (1.6074824584095362,12.508328849994291) node[anchor=north west] {$\vdots$};
\draw (8.306263322640009,12.508328849994291) node[anchor=north west] {$\vdots$};
\draw (10.806345224853041,11.302600523900632) node[anchor=north west] {$\cdots$};
\draw (10.806345224853041,9.176814866257862) node[anchor=north west] {$\cdots$};
\draw (10.806345224853041,2.493745947678496) node[anchor=north west] {$\cdots$};
\draw (0.,1.3333333333333333)-- (1.3333333333333333,1.3333333333333333);
\draw (1.3333333333333333,1.3333333333333333)-- (1.3333333333333333,0.);
\draw (0.,3.1111111111111107)-- (1.3333333333333333,3.1111111111111107);
\draw (1.3333333333333333,3.1111111111111107)-- (1.3333333333333333,2.6666666666666665);
\draw (1.3333333333333333,2.6666666666666665)-- (0.,2.6666666666666665);
\draw (0.,3.7037037037037033)-- (1.3333333333333333,3.7037037037037033);
\draw (1.3333333333333333,3.7037037037037037)-- (1.3333333333333333,3.5555555555555554);
\draw (1.3333333333333333,3.5555555555555554)-- (0.,3.5555555555555554);
\draw (2.6666666666666665,0.)-- (2.6666666666666665,1.3333333333333333);
\draw (2.6666666666666665,1.3333333333333333)-- (3.111111111111111,1.3333333333333333);
\draw (3.111111111111111,1.3333333333333333)-- (3.111111111111111,0.);
\draw (2.6666666666666665,2.6666666666666665)-- (2.6666666666666665,3.111111111111111);
\draw (2.6666666666666665,3.111111111111111)-- (3.111111111111111,3.111111111111111);
\draw (3.111111111111111,3.111111111111111)-- (3.111111111111111,2.6666666666666665);
\draw (3.111111111111111,3.111111111111111)-- (3.111111111111111,2.6666666666666665);
\draw (3.111111111111111,2.6666666666666665)-- (2.6666666666666665,2.6666666666666665);
\draw (2.6666666666666665,3.5555555555555554)-- (2.6666666666666665,3.7037037037037037);
\draw (2.6666666666666665,3.7037037037037037)-- (3.111111111111111,3.7037037037037037);
\draw (3.111111111111111,3.7037037037037037)-- (3.111111111111111,3.5555555555555554);
\draw (3.111111111111111,3.5555555555555554)-- (2.6666666666666665,3.5555555555555554);
\draw (0.,8.444444444444445)-- (1.3333333333333333,8.444444444444445);
\draw (1.3333333333333333,8.444444444444445)-- (1.3333333333333333,8.);
\draw (0.,9.037037037037036)-- (1.3333333333333333,9.037037037037036);
\draw (1.3333333333333333,9.037037037037036)-- (1.3333333333333333,8.88888888888889);
\draw (1.3333333333333333,8.88888888888889)-- (0.,8.88888888888889);
\draw (2.6666666666666665,8.)-- (2.6666666666666665,8.444444444444445);
\draw (2.6666666666666665,8.444444444444445)-- (3.111111111111111,8.444444444444445);
\draw (3.111111111111111,8.444444444444445)-- (3.111111111111111,8.);
\draw (2.6666666666666665,8.88888888888889)-- (2.6666666666666665,9.037037037037036);
\draw (2.6666666666666665,9.037037037037036)-- (3.111111111111111,9.037037037037036);
\draw (3.111111111111111,9.037037037037036)-- (3.111111111111111,8.88888888888889);
\draw (3.111111111111111,8.88888888888889)-- (2.6666666666666665,8.88888888888889);
\draw (8.,1.3333333333333333)-- (8.444444444444445,1.3333333333333333);
\draw (8.444444444444445,1.3333333333333333)-- (8.444444444444445,0.);
\draw (8.,3.111111111111111)-- (8.444444444444445,3.111111111111111);
\draw (8.444444444444445,3.111111111111111)-- (8.444444444444445,2.6666666666666665);
\draw (8.444444444444445,2.6666666666666665)-- (8.,2.6666666666666665);
\draw (8.,3.7037037037037037)-- (8.444444444444445,3.7037037037037037);
\draw (8.444444444444445,3.7037037037037037)-- (8.444444444444445,3.5555555555555554);
\draw (8.444444444444445,3.5555555555555554)-- (8.,3.5555555555555554);
\draw (8.88888888888889,0.)-- (8.88888888888889,1.3333333333333333);
\draw (8.88888888888889,1.3333333333333333)-- (9.037037037037036,1.3333333333333333);
\draw (9.037037037037036,1.3333333333333333)-- (9.037037037037036,0.);
\draw (8.88888888888889,2.6666666666666665)-- (8.88888888888889,3.111111111111111);
\draw (8.88888888888889,3.111111111111111)-- (9.037037037037036,3.111111111111111);
\draw (9.037037037037036,3.111111111111111)-- (9.037037037037036,2.6666666666666665);
\draw (9.037037037037036,2.6666666666666665)-- (8.88888888888889,2.6666666666666665);
\draw (2.814667436690525,1.0965691596188184) node[anchor=north west] {$\cdots$};
\draw (1.1473648923979489,4.767887371492285) node[anchor=north west] {\tiny{(1,1)}};
\draw (1.1473648923979489,8.205139562939789) node[anchor=north west] {\tiny{(1,2)}};
\draw (1.1473648923979489,10.908319892342731) node[anchor=north west] {\tiny{(1,3)}};
\draw (7.754976394068301,8.205139562939789) node[anchor=north west] {\tiny{(2,2)}};
\draw (7.754976394068301,4.767887371492285) node[anchor=north west] {\tiny{(2,1)}};
\draw (9.958549267550184,4.767887371492285) node[anchor=north west] {\tiny{(3,1)}};
\draw (7.754976394068301,10.908319892342731) node[anchor=north west] {\tiny{(2,3)}};
\draw (-2.763673061498972,1.2537422988200826) node[anchor=north west] {\tiny{(1,1)(1,1)}};
\draw (1.303673061498972,2.895279241279673) node[anchor=north west] {\tiny{(1,1)(2,2)}};
\draw (-2.763673061498972,3.395279241279673) node[anchor=north west] {\tiny{(1,1)(1,2)}};
\draw (9.958549267550184,8.205139562939789) node[anchor=north west] {\tiny{(3,2)}};
\draw (9.958549267550184,10.908319892342731) node[anchor=north west] {\tiny{(3,3)}};
\draw (-2.763673061498972,8.737854137636163) node[anchor=north west] {\tiny{(1,2)(1,1)}};
\end{tikzpicture}
\caption{Basic rectangles of the infinite affine transformations.}  \label{Fig1}
\end{figure}
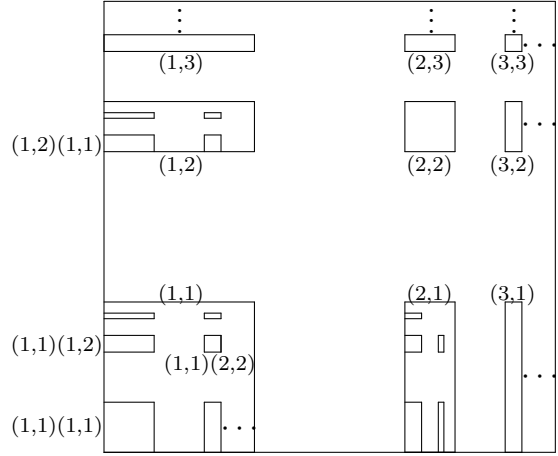

Since $P=\mathop{\sum}\limits_{i, j=1}^\infty p_{(i, j)} P\circ S_{(i, j)}^{-1}$, then, by induction, $P=\mathop\sum\limits_{\go \in \mathcal{I}^k} p_\go P\circ S_\go^{-1}$ for any $k\in \D N$.  Hence, we have the following statement:

\begin{lemma} \label{lemma1} Let $f: \D R^2 \to \D R^+$ be Borel measurable and $k\in \D N$. Then,
\[\int f \,dP=\sum_{\go \in \mathcal{I}^k}p_\go\int f\circ S_\go \,dP.\]
\end{lemma}

Let $S_{(i,j)}^{(1)}$ and $S_{(i,j)}^{(2)}$ be the horizontal and vertical components of the transformations $S_{(i,j)}.$ Then, for all $(x_1, x_2) \in \D R^2,$ we have
$S_{(i,j)}^{(1)}(x_1) =r^i x_1 +1-r^{i-1}$ and $S_{(i,j)}^{(2)}(x_2) =r^j x_2 +1-r^{j-1};$ hence, $S_{(i,j)}^{(1)}$ and $S_{(i,j)}^{(2)}$ are similarity mappings on $\D R$ with similarity ratios $s_{(i,j)}^{(1)}:=r^i$ and $s_{(i, j)}^{(2)}:=r^j , $ respectively. Similarly, for $\go=(i_1, j_1)(i_2, j_2)\cdots (i_k, j_k) \in \mathcal{I}^k$, $k\geq 1$, let $S_{\go}^{(1)}$ and $S_{\go}^{(2)}$ represent the horizontal and vertical components of the transformation $S_\go$ on $\D R^2$. Then,  $S_{\go}^{(1)}$ and $S_{\go}^{(2)}$ are similarity mappings on $\D R$ with similarity ratios $s_{\go}^{(1)}$ and $s_{\go}^{(2)},$ respectively, such that
$ S_{\go}^{(1)}=S_{(i_1, j_1)}^{(1)}\circ \cdots \circ S_{(i_k, j_k)}^{(1)}$  and $ S_{\go}^{(2)}=S_{(i_1, j_1)}^{(2)}\circ \cdots \circ S_{(i_k, j_k)}^{(2)}$ .
Thus, it follows that
\begin{align*}& s_{\go}^{(1)}=s_{(i_1, j_1)}^{(1)}s_{(i_2, j_2)}^{(1)}\cdots s_{(i_k, j_k)}^{(1)} =r^{i_1+i_2+\cdots+i_k} \te{ and }\\
&s_{\go}^{(2)}=s_{(i_1, j_1)}^{(2)}s_{(i_2, j_2)}^{(2)}\cdots s_{(i_k, j_k)}^{(2)} =r^{j_1+j_2+\cdots+j_k}.
\end{align*}
Moreover, we have $P(J_\go)=p_\go=p_{(i_1, j_1)}p_{(i_2, j_2)}\cdots p_{(i_k, j_k)}=\frac{1}{2^{i_1+i_2+\cdots+i_k+j_1+j_2+\cdots+j_k}}$.
Let $X:=(X_1, X_2)$ be a bivariate random variable with distribution $P$. Let $P_1, P_2$ be the marginal distributions of $P$, i.e., $P_1(A)=P(A\times \D R)=P\circ \pi_1^{-1} (A) $ for all $A \in \F B$, and $P_2(B)=P(\D R \times B)=P\circ \pi_2^{-1} (B)$ for all $B \in \F B$, where $\pi_1, \pi_2$ are projections given by $\pi_1(x_1, x_2)=x_1$ and $\pi_2 (x_1, x_2)=x_2$ for all $(x_1, x_2) \in \D R^2$. Here $\F B$ is the Borel $\gs$-algebra on $\D R$. Then, $X_1$ has distribution $P_1$ and $X_2$ has distribution $P_2$. Let $S_{(i, j)}^{-(1)}$ and $S_{(i, j)}^{-(2)}$ denote respectively the inverse images of the horizontal and vertical components of the transformations $S_{(i, j)}$ for all $i, j\in \D N$.
Then, the following lemma is known (\cite{H}, \cite{MaU}, \cite{M}):

\begin{lemma} \label{lemma2}  Let $P_1$ and $P_2$ be the marginal distributions of the probability measure $P$. Then,
\[P_1 =\sum_{i=1}^\infty \frac 1{2^i} P_1\circ S_{(i, j)}^{-(1)}  \te{  and } P_2 =\sum_{j=1}^\infty \frac 1{2^j} P_2\circ S_{(i, j)}^{-(2)}.\]
\end{lemma}

\begin{remark} \label{rem1} Since $S_{(i, j)}^{(1)}$ and $S_{(i, j)}^{(2)}$ are similarity mappings, from Lemma~\ref{lemma2}, one can see that both the marginal distributions $P_1$ and $P_2$ are self-similar measures on $\D R$ generated by an infinite collection of similarities associated with the probability vector $(\frac 1 2, \frac 1 {2^2}, \cdots)$.

\end{remark}

\begin{lemma} \label{lemma3} Let $E(X)$ and $V(X)$ denote the expectation and the variance of the random variable $X$. Then, \[E(X)=(E(X_1), \, E(X_2))=(\frac 12, \frac 12) \te{ and } V:=V(X)=E\|X-(\frac 12, \frac  12 )\|^2=\frac {1}{4}.\]
\end{lemma}

\begin{proof} By Lemma~\ref{lemma2},
$P_1=P_2=\mu$, where $\mu$ is a unique Borel probability measure on $\D R$ such that
\[\mu=\sum_{k=1}^\infty \frac 1{2^k} \mu \circ S_{(k,j)}^{-(1)}=
\sum_{k=1}^\infty \frac 1{2^k} \mu \circ S_{(i,k)}^{-(2)}. \]
Hence, $X_1=X_2$,  and by \cite[Lemma~2.2]{R1},
$E(X_1)=E(X_2)=\frac 12, $ and $ V(X_1)=V(X_2)=\frac {1}{8},$
which implies that
$ E\|X-(\frac 12, \frac 12)\|^2=E(X_1-\frac 12)^2 +E(X_2-\frac 12)^2=V(X_1)+V(X_2)=\frac {1}{4}. $
\end{proof}

\begin{remark} Using the standard rule of probability, for any $(a, b) \in \D R^2$, we have $ E\|X-(a, b)\|^2=V+\|(a, b)-(\frac 12, \frac 12)\|^2$, which yields that the optimal set of one-mean consists of the expected value and the corresponding quantization error is the variance $V$ of the random variable $X$.
\end{remark}

\begin{lemma} \label{lemma4}
Let $\go \in \mathcal{I}^\ast$. Then,

$(i)$ $ E(X | X \in J_{\go(\infty, \infty)})=S_{\go^-(\go^{(1)}_{|\go|}+1,\,  \go^{(2)}_{|\go|}+1)}(\frac 12, \frac 12)+
(s_\go^{(1)} \frac{1}{2}(1-r), s_\go^{(2)}\frac 12 (1-r) );$

$(ii)$ $E(X | X \in J_{\go(\es, \infty)})=S_{\go^-(\go^{(1)}_{|\go|},\,  \go^{(2)}_{|\go|}+1)}(\frac 12, \frac 12)+(0,\,
s_\go^{(2)}\frac 12 (1-r) ),$ and

$(iii)$ $E(X | X \in J_{\go(\infty, \es)})=S_{\go^-(\go^{(1)}_{|\go|}+1,\,  \go^{(2)}_{|\go|})}(\frac 12, \frac 12)+(s_\go^{(1)} \frac{1}{2}(1-r), \, 0).$

\end{lemma}

\begin{proof} First prove $(i)$.  Since $P(J_{\go(\infty, \infty)})=p_{\go(\infty, \infty)}=p_\go$ and $p_{\go^-(\go^{(1)}_{|\go|}+i,\,  \go^{(2)}_{|\go|}+j)}=p_\go \frac 1{2^{i+j}},$
\begin{align*} &E(X | X \in J_{\go(\infty, \infty)})=E(X |X \in \mathop{\uu}\limits_{i, j=1}^\infty J_{\go^-(\go^{(1)}_{|\go|}+i,\,  \go^{(2)}_{|\go|}+j)})\\
&=\frac{1}{P(J_{\go(\infty, \infty)})}\sum_{i, j=1}^\infty p_{\go^-(\go^{(1)}_{|\go|}+i,\,  \go^{(2)}_{|\go|}+j)} S_{\go^-(\go^{(1)}_{|\go|}+i,\,  \go^{(2)}_{|\go|}+j)}(\frac 12, \frac 12)=\sum_{i, j=1}^\infty \frac 1{2^{i+j}}S_{\go^-(\go^{(1)}_{|\go|}+i,\,  \go^{(2)}_{|\go|}+j)}(\frac 12, \frac 12).
\end{align*}
Notice that
\begin{align*}
&S_{\go^-(\go^{(1)}_{|\go|}+i,\,  \go^{(2)}_{|\go|}+j)}(\frac 12, \frac 12)-S_{\go^-(\go^{(1)}_{|\go|}+1,\,  \go^{(2)}_{|\go|}+1)}(\frac 12, \frac 12)\\
&=\Big(S^{(1)}_{\go^-(\go^{(1)}_{|\go|}+i,\,  \go^{(2)}_{|\go|}+j)}(\frac 12), \, S^{(2)}_{\go^-(\go^{(1)}_{|\go|}+i,\,  \go^{(2)}_{|\go|}+j)}(\frac 12)\Big)-\Big(S^{(1)}_{\go^-(\go^{(1)}_{|\go|}+1,\,  \go^{(2)}_{|\go|}+1)}(\frac 12), S^{(2)}_{\go^-(\go^{(1)}_{|\go|}+1,\,  \go^{(2)}_{|\go|}+1)}(\frac 12)\Big)\notag\\
&=\Big(S^{(1)}_{\go^-(\go^{(1)}_{|\go|}+i,\,  \go^{(2)}_{|\go|}+j)}(\frac 12) -S^{(1)}_{\go^-(\go^{(1)}_{|\go|}+1,\,  \go^{(2)}_{|\go|}+1)}(\frac 12), \, S^{(2)}_{\go^-(\go^{(1)}_{|\go|}+i,\,  \go^{(2)}_{|\go|}+j)}(\frac 12) -S^{(2)}_{\go^-(\go^{(1)}_{|\go|}+1,\,  \go^{(2)}_{|\go|}+1)}(\frac 12) \Big ).\notag
\end{align*}
Since
\begin{align*}
&S^{(1)}_{\go^-(\go^{(1)}_{|\go|}+i,\,  \go^{(2)}_{|\go|}+j)}(\frac 12) -S^{(1)}_{\go^-(\go^{(1)}_{|\go|}+1,\,  \go^{(2)}_{|\go|}+1)}(\frac 12)=s_{\go^-}^{(1)} \Big(S^{(1)}_{(\go^{(1)}_{|\go|}+i,\,  \go^{(2)}_{|\go|}+j)}(\frac 12) -S^{(1)}_{(\go^{(1)}_{|\go|}+1,\,  \go^{(2)}_{|\go|}+1)}(\frac 12)\Big)\\
&=s_{\go^-}^{(1)} \Big(r^{\go^{(1)}_{|\go|}+i} (\frac 12) -r^{\go^{(1)}_{|\go|}+i-1}-r^{\go^{(1)}_{|\go|}+1} (\frac 12)+r^{\go^{(1)}_{|\go|}+1-1}\Big)=s_\go^{(1)}\Big(\frac 12 r^i-r^{i-1}-\frac{r}{2}+1 \Big)\\
&=s_\go^{(1)}(1-\frac{r}{2})(1- r^{i-1}), \te{ and similarly}
\end{align*}
$S^{(2)}_{\go^-(\go^{(1)}_{|\go|}+i,\,  \go^{(2)}_{|\go|}+j)}(\frac 12) -S^{(2)}_{\go^-(\go^{(1)}_{|\go|}+1,\,  \go^{(2)}_{|\go|}+1)}(\frac 12)=s_\go^{(2)}(1-\frac{r}{2})(1- r^{j-1}).$ Hence, we have that\\
$S_{\go^-(\go^{(1)}_{|\go|}+i,\,  \go^{(2)}_{|\go|}+j)}(\frac 12, \frac 12)=S_{\go^-(\go^{(1)}_{|\go|}+1,\,  \go^{(2)}_{|\go|}+1)}(\frac 12, \frac 12)+(s_\go^{(1)}(u) ), s_\go^{(2)}(v)),$
where $u=(\frac{2-r}{2})(1- r^{i-1}) $ and $ v=(\frac{2-r}{2})(1- r^{j-1}).$  Therefore,
\begin{align*} &E(X | X \in J_{\go(\infty, \infty)})= S_{\go^-(\go^{(1)}_{|\go|}+1,\,  \go^{(2)}_{|\go|}+1)}(\frac 12, \frac 12)+\sum_{i, j=1}^\infty \frac 1{2^{i+j}}(s_\go^{(1)}(u), s_\go^{(2)}(v)) \\
&=S_{\go^-(\go^{(1)}_{|\go|}+1,\,  \go^{(2)}_{|\go|}+1)}(\frac 12, \frac 12)+(s_\go^{(1)} (\frac{1-r}{2}), s_\go^{(2)} (\frac{1-r}{2}) ).
%\\
%&=S_{\go^-(\go^{(1)}_{|\go|}+1,\,  \go^{(2)}_{|\go|}+1)}(\frac %12, \frac 12)+(s^{(1)}_{\go^-(\go^{(1)}_{|\go|}+1,\,  \go^{(2)}%_{|\go|}+1)},\, s^{(2)}_{\go^-(\go^{(1)}_{|\go|}+1,\,  %\go^{(2)}_{|\go|}+1)}).
\end{align*}
Proofs of (ii) and (iii) are similar.
\end{proof}

\begin{note} \label{note1}

For words $\gb, \gg, \cdots, \gd$ in $\mathcal{I}^\ast$, by $a(\gb, \gg, \cdots, \gd)$ we denote the conditional expectation of the random variable $X$ given $J_\gb\uu J_\gg \uu\cdots \uu J_\gd,$ i.e.,
\begin{equation}  \label{eq00} a(\gb, \gg, \cdots, \gd)=E(X|X\in J_\gb \uu J_\gg \uu \cdots \uu J_\gd)=\frac{1}{P(J_\gb\uu \cdots \uu J_\gd)}\int_{J_\gb\uu \cdots \uu J_\gd} (x_1, x_2) dP.
\end{equation}
Then, for $\go\in \mathcal{I}^\ast$,
\begin{align} \label{eq1} \left\{ \begin{array}{ll} a(\go)=S_\go(E(X))=S_\go(\frac 12, \frac 12), & \\  a(\go(\es, \infty))=E(X | X \in J_{\go(\es, \infty)}),& \\
a(\go(\infty, \es))=E(X | X \in J_{\go(\infty, \es)}), \te{ and } & \\ a(\go (\infty, \infty))=E(X | X \in J_{\go(\infty, \infty)}).&
 \end{array} \right.
 \end{align}
 Thus, by Lemma~\ref{lemma4}, if $\go=(1,1)$, then $a((1,1))=(\frac{r}{2}, \frac{r}{2})$, $a((1, 1)(\infty, \es))=(1-\frac{r}{2}, \frac{r}{2})$, $a((1, 1)(\es, \infty))=(\frac{r}{2},1- \frac{r}{2})$, and $a((1, 1)(\infty, \infty))=(1- \frac{r}{2},1- \frac{r}{2} )$. In addition,
 \begin{align} \label{eq32} \left\{\begin{array} {ll}
 a((1,1), (1, 1)(\infty, \es))=(\frac 12, \frac{r}{2}),  & \\ a((1, 1)(\es, \infty), (1, 1)(\infty, \infty))=(\frac 12, 1- \frac{r}{2}), & \\
 a((1,1), (1, 1)(\es, \infty))=(\frac{r}{2}, \frac 12),  & \\  a((1, 1)(\infty, \es), (1, 1)(\infty, \infty))=(1- \frac{r}{2}, \frac 12).
 \end{array}\right.
 \end{align}
Moreover, for $\go \in \mathcal{I}^k$, $k\geq 1$, it is easy to see that
\begin{align} \label{eq2}  &\int_{J_\go}\|x-(a, b)\|^2 dP= p_\go \int\|(x_1, x_2) -(a, b)\|^2 dP\circ S_\go^{-1}\\
&=p_\go \Big(s_{\go}^{(1)2} V(X_1)+{s_{\go}^{(2)2}} V(X_2)+\|S_\go(\frac 12, \frac 12)-(a, b)\|^2\Big), \notag
\end{align}
where $s_{\go}^{(k)2}:=(s_{\go}^{(k)})^2$ for $k=1, 2$.
The expressions \eqref{eq1} and \eqref{eq2}  are useful to obtain the optimal sets and the corresponding quantization errors with respect to the probability distribution $P$.
\end{note}

For the rest of the article $r=\frac{1}{3} $ is assumed, which is the most important case due to its intimate connection with the standard Cantor system.

\section{Optimal sets of $n$-means for $n=2, 3$}

In the this section, we determine the optimal sets of two- and three-means, and their quantization errors.

\begin{lemma} \label{lemma5} Let $P$ be the affine measure on $\D R^2$ and let $\go \in \mathcal{I}^\ast$. Then,
\begin{align*}
 &\int_{J_{\go(\infty, \infty)}} \|x-a(\go(\infty, \infty))\|^2  dP= \int_{J_{\go(\es, \infty)}} \|x-a(\go(\es, \infty))\|^2  dP\\
 &=\int_{J_{\go(\infty, \es)}} \|x-a(\go(\infty, \es))\|^2  dP=\int_{J_\go} \|x-a(\go)\|^2 dP=p_\go(s_{\go}^{(1)2}+{s_{\go}^{(2)2}})\frac 18.
\end{align*}
\end{lemma}

\begin{proof} Let us first prove $\int_{J_{\go(\infty, \infty)}} \|x-a(\go(\infty, \infty))\|^2  dP=p_\go(s_{\go}^{(1)2}+{s_{\go}^{(2)2}})\frac 18$.
By Lemma~\ref{lemma4}, we have
\begin{align} \label{eq34}
&\int_{J_{\go(\infty, \infty)}} \|x-a(\go(\infty, \infty))\|^2  dP=\sum_{i, j=1}^\infty \int_{J_{\go^-(\go^{(1)}_{|\go|}+i, \, \go^{(2)}_{|\go|}+j)}} \|x-a(\go(\infty, \infty))\|^2  dP \\
&=p_\go\sum_{i, j=1}^\infty\frac 1{2^{i+j}} \int  \|S_{\go^-(\go^{(1)}_{|\go|}+i, \, \go^{(2)}_{|\go|}+j)}(x_1, x_2) -S_{\go^-(\go^{(1)}_{|\go|}+1,\,  \go^{(2)}_{|\go|}+1)}(\frac 12, \frac 12)\notag\\
&\qquad \qquad \qquad -(s^{(1)}_{\go^-(\go^{(1)}_{|\go|}+1,\,  \go^{(2)}_{|\go|}+1)},\, s^{(2)}_{\go^-(\go^{(1)}_{|\go|}+1,\,  \go^{(2)}_{|\go|}+1)})\|^2  dP.\notag
\end{align}
Note that $S_{\go^-(\go^{(1)}_{|\go|}+i, \, \go^{(2)}_{|\go|}+j)}(x_1, x_2)=\Big(S^{(1)}_{\go^-(\go^{(1)}_{|\go|}+i, \, \go^{(2)}_{|\go|}+j)}(x_1), \, S^{(2)}_{\go^-(\go^{(1)}_{|\go|}+i, \, \go^{(2)}_{|\go|}+j)}(x_2)\Big)$ and

 $S_{\go^-(\go^{(1)}_{|\go|}+1, \, \go^{(2)}_{|\go|}+1)}(\frac 12, \frac 12)=\Big(S^{(1)}_{\go^-(\go^{(1)}_{|\go|}+1, \, \go^{(2)}_{|\go|}+1)}(\frac 12), \, S^{(2)}_{\go^-(\go^{(1)}_{|\go|}+1, \, \go^{(2)}_{|\go|}+1)}(\frac 12)\Big)$. Moreover, we have
 \begin{align*}
 &\Big(S^{(1)}_{\go^-(\go^{(1)}_{|\go|}+i, \, \go^{(2)}_{|\go|}+j)}(x_1)-S^{(1)}_{\go^-(\go^{(1)}_{|\go|}+1, \, \go^{(2)}_{|\go|}+1)}(\frac 12)-s^{(1)}_{\go^-(\go^{(1)}_{|\go|}+1,\,  \go^{(2)}_{|\go|}+1)}\Big)^2\\
 &=s_{\go^-}^{(1)2}\Big(S^{(1)}_{(\go^{(1)}_{|\go|}+i, \, \go^{(2)}_{|\go|}+j)}(x_1)-S^{(1)}_{(\go^{(1)}_{|\go|}+1, \, \go^{(2)}_{|\go|}+1)}(\frac 12)-s^{(1)}_{(\go^{(1)}_{|\go|}+1,\,  \go^{(2)}_{|\go|}+1)}\Big)^2\\
 &=s_{\go^-}^{(1)2}\Big(\Big(S^{(1)}_{(\go^{(1)}_{|\go|}+i, \, \go^{(2)}_{|\go|}+j)}(x_1)-S^{(1)}_{(\go^{(1)}_{|\go|}+i, \, \go^{(2)}_{|\go|}+j)}(\frac 12)\Big)+\Big(S^{(1)}_{(\go^{(1)}_{|\go|}+i, \, \go^{(2)}_{|\go|}+j)}(\frac 12)-S^{(1)}_{(\go^{(1)}_{|\go|}+1, \, \go^{(2)}_{|\go|}+1)}(\frac 12)\\
 &\qquad \qquad \qquad -s^{(1)}_{(\go^{(1)}_{|\go|}+1,\,  \go^{(2)}_{|\go|}+1)}\Big)\Big)^2.
 \end{align*}
 Now break the above expression using the square formula and note the fact that
\begin{align*}
&\int \Big(S^{(1)}_{(\go^{(1)}_{|\go|}+i, \, \go^{(2)}_{|\go|}+j)}(x_1)-S^{(1)}_{(\go^{(1)}_{|\go|}+i, \, \go^{(2)}_{|\go|}+j)}(\frac 12)\Big)^2 dP_1=s^{(1)2}_{(\go^{(1)}_{|\go|}+i, \, \go^{(2)}_{|\go|}+j)}V(X_1)=s^{(1)2}_{(\go^{(1)}_{|\go|}, \, \go^{(2)}_{|\go|})}\frac 1{9^i}\frac 18, \text{and} \\
&\int \Big(S^{(1)}_{(\go^{(1)}_{|\go|}+i, \, \go^{(2)}_{|\go|}+j)}(\frac 12)-S^{(1)}_{(\go^{(1)}_{|\go|}+i, \, \go^{(2)}_{|\go|}+j)}(\frac 12)\Big) dP_1=0, \te{ and after some simplification } \text{we have} \\
&\Big(S^{(1)}_{(\go^{(1)}_{|\go|}+i, \, \go^{(2)}_{|\go|}+j)}(\frac 12)-S^{(1)}_{(\go^{(1)}_{|\go|}+1, \, \go^{(2)}_{|\go|}+1)}(\frac 12)-s^{(1)}_{(\go^{(1)}_{|\go|}+1,\,  \go^{(2)}_{|\go|}+1)}\Big)^2=s^{(1)2}_{(\go^{(1)}_{|\go|}, \, \go^{(2)}_{|\go|})}\frac 1 4 (1-\frac 5{3^i})^2.
\end{align*}
Thus, it follows that
\begin{align*}&\int \Big(S^{(1)}_{\go^-(\go^{(1)}_{|\go|}+i, \, \go^{(2)}_{|\go|}+j)}(x_1)-S^{(1)}_{\go^-(\go^{(1)}_{|\go|}+1, \, \go^{(2)}_{|\go|}+1)}(\frac 12)-s^{(1)}_{\go^-(\go^{(1)}_{|\go|}+1,\,  \go^{(2)}_{|\go|}+1)}\Big)^2 dP_1\\
&=s^{(1)2}_{\go}\Big(\frac 1{9^i}\frac 18+\frac 1 4 (1-\frac 5{3^i})^2\Big), \te{ and similarly}\\
& \int \Big(S^{(2)}_{\go^-(\go^{(2)}_{|\go|}+i, \, \go^{(2)}_{|\go|}+j)}(x_2)-S^{(2)}_{\go^-(\go^{(1)}_{|\go|}+1, \, \go^{(2)}_{|\go|}+1)}(\frac 12)-s^{(2)}_{\go^-(\go^{(1)}_{|\go|}+1,\,  \go^{(2)}_{|\go|}+1)}\Big)^2 dP_2\\
&=s^{(2)2}_{\go}\Big(\frac 1{9^j}\frac 18+\frac 1 4 (1-\frac 5{3^j})^2\Big).\ \text{Therefore, \eqref{eq34} implies that}
\end{align*}
%Therefore, \eqref{eq34} implies that
\begin{align*}
&\int_{J_{\go(\infty, \infty)}} \|x-a(\go(\infty, \infty))\|^2  dP\\
&=p_\go \sum_{i, j=1}^\infty \frac {1}{2^{i+j}} \Big(s^{(1)2}_{\go}\Big(\frac 1{9^i}\frac 18+\frac 1 4 (1-\frac 5{3^i})^2\Big) + s^{(2)2}_{\go}\Big(\frac 1{9^j}\frac 18+\frac 1 4 (1-\frac 5{3^j})^2\Big)\Big) =p_\go(s_{\go}^{(1)2}+{s_{\go}^{(2)2}})\frac 18.
\end{align*}
Other equalities of the statement are proved similarly.
\end{proof}

\begin{lemma} \label{lemma67}  Let $P$ be the affine measure on $\D R^2$, and let $\set{(a, p), (b, p)}$ be a set of two points lying on the line $x_2=p$ for which the distortion error is smallest. Then, $a=\frac 16$, $b=\frac 56$, $p=\frac 12$ and the distortion error is $\frac 5{36}.$
\end{lemma}
%\frac 5{36}=0.138889
\begin{proof} Let $\gb=\set{(a, p), (b, p)}$.  Since the points for which the distortion error is smallest are the centroids of their own Voronoi regions, by the properties of centroids, we have
\[(a, p) P(M((a, p)|\gb))+(b, p) P(M((b, p)|\gb))=(\frac 12, \frac 12),\]
which implies $p P(M((a, p)|\gb))+p P(M((b, p)|\gb))=\frac 12$, i.e, $p=\frac 12$. Thus, the boundary of the Voronoi regions is the line $x_1=\frac 12$. Now, using the definition of conditional expectation,
\begin{align*}
&(a, \frac 12)=E(X : X \in M((a, \frac 12)|\gb))=E(X : X\in \mathop{\uu}\limits_{j=1}^\infty J_{(1, j)})=\frac 1{\sum_{j=1}^\infty p_{(1,j)}} \sum_{j=1}^\infty p_{(1, j)} S_{(1, j)}(\frac 12, \frac 12),
\end{align*}
which implies $(a, \frac 12)=(\frac 16, \frac 12)$ yielding $a=\frac 16$. Similarly, $b=\frac 56$. Then, the distortion error is
\[\int\min_{c\in \gb} \|x-c\|^2 dP=\mathop{\int}\limits_{\mathop{\uu}\limits_{j=1}^\infty J_{(1, j)}}\|x-(\frac 16, \frac 12)\|^2 dP+\mathop{\int}\limits_{\mathop{\uu}\limits_{i=2,j=1}^\infty J_{(i, j)}}\|x-(\frac 56, \frac 12)\|^2 dP=\frac 5{72}+\frac 5{72}=\frac 5{36}.\]
This completes the proof the lemma.
\end{proof}

The following lemma provides us information on where to look for points
of an optimal set of two-means.

\begin{lemma} \label{lemma68}
Let $P$ be the affine measure on $\D R^2$. The points in an optimal set of two-means can not lie on an oblique line of the affine set.
\end{lemma}

\begin{proof}
In the affine set, among all the oblique lines that pass through the point $(\frac 12, \frac 12)$, the line $x_2=x_1$ has the maximum symmetry, i.e., with respect to the line $x_2=x_1$ the affine set is geometrically symmetric. Also,
observe that, if two basic rectangles of similar geometrical shape lie in the opposite sides of the line $x_2=x_1$, and are equidistant from the line $x_2=x_1$, then they have the same probability (see Figure~\ref{Fig1}); hence, they are symmetric with respect to the probability distribution $P$. Due to this, among all the pairs of two points which have the boundaries of the Voronoi regions oblique lines passing through the point $(\frac 12, \frac 12)$, the two points which have the boundary of the Voronoi regions the line $x_2=x_1$ will give the smallest distortion error. Again, we know the two points which give the smallest distortion error are the centroids of their own Voronoi regions. Let $(a_1, b_1)$ and $(a_2, b_2)$ be the centroids of the left half and the right half of the affine set with respect to the line $x_2=x_1$ respectively. Then, from the definition of conditional expectation, we have

\begin{align*}
&(a_1, b_1)=2\Big[ \mathop{\sum}\limits_{i=1, j=i+1}^\infty \frac 1{2^{i+j}} S_{(i, j)}(\frac 12, \frac 12)+\mathop{\sum}\limits_{k_1=1}^\infty \mathop{\sum}\limits_{\mathop{i=1}\limits_{j=i+1}}^\infty \frac 1{2^{2k_1+i+j}} S_{(k_1, k_1)(i, j)}(\frac 12, \frac 12)\\
&+\mathop{\sum}\limits_{k_1=1}^\infty \mathop{\sum}\limits_{k_2=1}^\infty \mathop{\sum}\limits_{\mathop{i=1}\limits_{j=i+1}}^\infty \frac 1{2^{2k_1+2k_2+i+j}} S_{(k_1, k_1)(k_2, k_2)(i, j)}(\frac 12, \frac 12)\\
&+\mathop{\sum}\limits_{k_1=1}^\infty \mathop{\sum}\limits_{k_2=1}^\infty  \mathop{\sum}\limits_{k_3=1}^\infty \mathop{\sum}\limits_{\mathop{i=1}\limits_{j=i+1}}^\infty \frac 1{2^{2k_1+2k_2+2k_3+i+j}} S_{(k_1, k_1)(k_2, k_2)(k_3, k_3)(i, j)}(\frac 12, \frac 12)+\cdots\Big]=(\frac 3 {10}, \frac{7}{10}),
\end{align*}
and
\begin{align*}
&(a_2, b_2)=2\Big(\mathop{\sum}\limits_{i=1}^\infty \mathop{\sum}\limits_{j=1}^{i-1}\frac 1{2^{i+j}} S_{(i, j)}(\frac 12, \frac 12)+\mathop{\sum}\limits_{k_1=1}^\infty \mathop{\sum}\limits_{i=1}^\infty \mathop{\sum}\limits_{j=1}^{i-1}\frac 1{2^{2k_1+i+j}} S_{(k_1, k_1)(i, j)}(\frac 12, \frac 12)\\
&+\mathop{\sum}\limits_{k_1=1}^\infty \mathop{\sum}\limits_{k_2=1}^\infty \mathop{\sum}\limits_{i=1}^\infty \mathop{\sum}\limits_{j=1}^{i-1} \frac 1{2^{2k_1+2k_2+i+j}} S_{(k_1, k_1)(k_2, k_2)(i, j)}(\frac 12, \frac 12)\\
&+\mathop{\sum}\limits_{k_1=1}^\infty \mathop{\sum}\limits_{k_2=1}^\infty  \mathop{\sum}\limits_{k_3=1}^\infty \mathop{\sum}\limits_{i=1}^\infty \mathop{\sum}\limits_{j=1}^{i-1} \frac 1{2^{2k_1+2k_2+2k_3+i+j}} S_{(k_1, k_1)(k_2, k_2)(k_3, k_3)(i, j)}(\frac 12, \frac 12)+\cdots\Big)=(\frac 7 {10}, \frac{3}{10}).
\end{align*}
Let $\gb=\set{(\frac 3 {10}, \frac{7}{10}), (\frac 7 {10}, \frac{3}{10})}$. Then, due to symmetry,
\begin{align*}
\int\min_{c\in \gb}\|x-c\|^2 dP=2 \int_{M((\frac 3 {10}, \frac{7}{10})|\gb)}\|x-(\frac 3 {10}, \frac{7}{10})\|^2 dP.
\end{align*}
Write
\begin{align*}&A:=(\mathop{\uu}\limits_{j=2}^4J_{(1,1)(1,1)(1,1)(1,1)(1, j)})\uu(\mathop\uu\limits_{j=2}^6J_{(1,1)(1,1)(1,1)(1, j)})\uu (\mathop\uu\limits_{j=3}^5J_{((1,1)(1,1)(1,1)(2, j)})\uu (\mathop{\uu}\limits_{j=2}^8J_{(1,1)(1,1)(1, j)})\\
&\uu (\mathop{\uu}\limits_{j=3}^6J_{(1,1)(1,1)(2, j)})\uu J_{(1,1)(1,1)(3,4)}\uu (\mathop{\uu}\limits_{j=2}^8J_{(1,1)(1, j)})\uu (\mathop{\uu}\limits_{j=3}^7J_{(1,1)(2, j)})\uu (\mathop{\uu}\limits_{j=4}^6J_{(1,1)(3, j)})\uu(\mathop{\uu}\limits_{j=2}^{10}J_{(1, j)})\\
&\uu(\mathop{\uu}\limits_{j=3}^{10}J_{(2, j)})\uu(\mathop{\uu}\limits_{j=4}^{10}J_{(3, j)})\uu(\mathop{\uu}\limits_{j=5}^{9}J_{(4, j)})\uu(\mathop{\uu}\limits_{j=6}^{7}J_{(5, j)}).
\end{align*}
Since $A$ is a proper subset of $M((\frac 3 {10}, \frac{7}{10})|\gb)$, we have
$\int\min_{c\in \gb}\|x-c\|^2 dP>2\mathop\int\limits_A\|x-(\frac 3 {10}, \frac{7}{10})\|^2 dP.
$  Now using \eqref{eq2}, and then upon simplification,  it follows that
\[\int\min_{c\in \gb}\|x-c\|^2 dP>2\mathop\int\limits_A\|x-(\frac 3 {10}, \frac{7}{10})\|^2 dP=0.13899,\]
which is larger than the distortion error $\frac{5}{36}$ obtained in Lemma~\ref{lemma67}. Hence, the points in an optimal set of two-means can not lie on a oblique line of the affine set. Thus, the assertion of the lemma follows.
\end{proof}

\begin{figure}
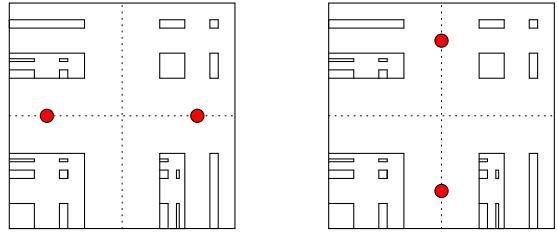


\caption{Optimal sets of two-means.}  \label{Fig2}
\end{figure}

\begin{figure}
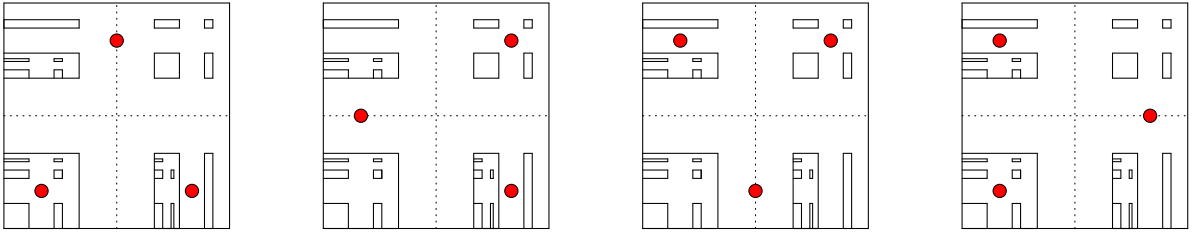

%
\caption{Optimal sets of three-means.}  \label{Fig3}
\end{figure}

\begin{figure}
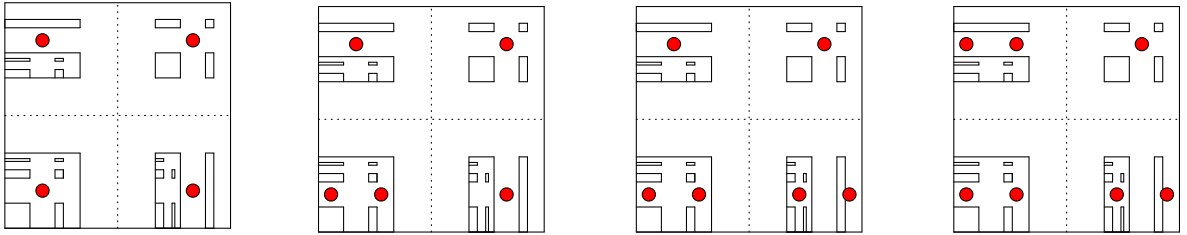

%
\caption{Optimal sets of n-means for $4\leq n\leq 7$. Optimal set of 4-means is unique; on the other hand, optimal sets of $n$-means for $n=5, 6, 7$ are not unique.}  \label{Fig4}
\end{figure}

%The following proposition gives the optimal sets of two-means.

\begin{prop} \label{prop1}
\tit {Let $P$ be the affine measure on $\D R^2$. Then, the sets $\set{(\frac 16, \frac 12), (\frac 56, \frac 12)}$   and  $\set{(\frac 12, \frac 16), (\frac 12, \frac 56)}$ form two different optimal sets of two-means with quantization error $\frac 5{36}.$}
\end{prop}

\begin{proof} By Lemma~\ref{lemma68}, it is known that the points in an optimal set of two-means can not lie on an oblique line of the affine set. Thus, by Lemma~\ref{lemma67}, we see that $\set{ (\frac 16, \frac 12), (\frac 56, \frac 12)}$ forms an optimal set of two-means with quantization error $\frac{5}{36}$. Due to symmetry, $\set{(\frac 12, \frac 16), (\frac 12, \frac 56)}$ forms another optimal set of two-means (see Figure~\ref{Fig2}); thus, the assertion follows.
\end{proof}

%The following proposition gives an optimal set of three-means.

\begin{prop}\label{prop2}
{\it Let $P$ be the affine measure on $\D R^2$. Then, the set $\set{(\frac 16, \frac 16), (\frac 56, \frac 16), (\frac 12, \frac 56)}$ forms an optimal set of three-means with quantization error $\frac 1{12}$.}
\end{prop}

\begin{proof} Let us first consider a three-point set $\gb$ given by $\gb=\set{(\frac 16, \frac 16), (\frac 56, \frac 16), (\frac 12, \frac 56)}$.
Then, using Lemma~\ref{lemma5} and equation \eqref{eq2}, we have
\begin{align*}
&\int \min_{a \in \beta}\|x-a\|^2  dP=\int_{J_{(1,1)}} \|x-(\frac 16, \frac 16)\|^2dP+\int_{J_{(1,1)(\infty, \es)}} \|x-(\frac 56, \frac 16)\|^2dP\\
&\qquad +\int_{J_{(1,1)(\es, \infty)}\uu J_{(1,1)(\infty, \infty)}} \|x-(\frac 12, \frac 56)\|^2dP=\frac 1{12}.
\end{align*}
Since $V_3$ is the quantization error for an optimal set of three-means, we have $\frac 1{12}\geq V_3$. Let $\ga=\set{(a_i, b_i) : 1\leq i\leq 3}$ be an optimal set of three-means. Since the optimal points are the centroids of their own Voronoi regions, we have $\ga\sci [0, 1]\times [0, 1]$. Let $A_1=[0, \frac 13]\times [0, \frac 13]$, $A_2=[\frac 2 3, 1]\times [0, \frac 13]$, $A_3=[0, \frac 13]\times [\frac 23, 1]$, and $A_4=[\frac 23, 1]\times [\frac 23, 1]$. Note that the centroids of $A_1$, $A_2$, $A_3$ and $A_4$ with respect to the probability distribution $P$ are respectively $(\frac 16, \frac 16)$, $(\frac 56, \frac 16)$, $(\frac 16, \frac 56)$ and $(\frac 56, \frac 56)$.  Suppose that $\ga$ does not contain any point from $\mathop{\uu}\limits_{i=1}^4 A_i$. Then, we can assume that all the points of $\ga$ are on the line $x_2=\frac 12$, i.e., $\ga=\set{(a_i, \frac 12): 1\leq i\leq 3}$ with $a_1< a_2< a_3$. If $a_1>\frac 13 ,$ quantization error can be strictly reduced by moving the point $(a_1, \frac 12)$ to $(\frac 13, \frac 12)$. So, we can assume that $a_1\leq \frac 13$. Similarly, we can show that $a_3\geq \frac 23$. Now, if $a_2<\frac 13$, then $A_3\uu A_4\sci M((a_3, \frac 12)|\ga)$. Moreover, for any $x=(x_1, x_2) \in J_{(1,1)(1,1)} \uu J_{(1,3)}$, we have $m(x):=\min_{c\in\ga}\|(x_1, x_2)-c\|^2 \geq (\frac 7{18})^2$ and so by \eqref{eq2} and Lemma~\ref{lemma5}, we obtain
\begin{align*}
&\int m(x)^2  dP=\mathop{\int}\limits_{J_{(1,1)(1,1)} \uu J_{(1,3)}} m(x)^2dP+\mathop{\int}\limits_{J_{(1,1)(\infty, \es)}\uu J_{(1,1)(\infty, \infty)}} m(x)^2dP\\
&\geq\frac 1{16} \Big((\frac 1{81}+\frac 1 {81})\frac 18+(\frac 7{18})^2\Big)+\frac 1{16} \Big((\frac{1}{9}+\frac{1}{27^2})\frac 18+(\frac 7{18})^2\Big)+\mathop{\int}\limits_{J_{(1,1)(\infty, \es)}\uu J_{(1,1)(\infty, \infty)}}\|x-(\frac 56, \frac 12)\|^2dP\\
&=\frac 1{16} \Big((\frac 1{81}+\frac 1 {81})\frac 18+(\frac 7{18})^2\Big)+\frac 1{16} \Big((\frac{1}{9}+\frac{1}{27^2})\frac 18+(\frac 7{18})^2\Big)+\frac{5}{72}=\frac{1043}{11664}>V_3,
\end{align*}
which is a contradiction, and so $a_2\geq \frac 13$ must be true. If $a_2>\frac 23$, similarly we can show that a contradiction arises. So, $\frac 13 < a_2<\frac 23$. Next, suppose that $\frac 12 \leq a_2< \frac 23$. Then, we have $\frac 12 (a_1+a_2)\leq \frac 13$ which implies $a_1\leq \frac 16$, for otherwise quantization error can be strictly reduced by moving $a_2$ to $(\frac 23, \frac 12)$, contradicting the fact that $\ga$ is an optimal set. Then, $\mathop{\uu}\limits_{j=1}^\infty J_{(1,1)(1, j)}\uu\mathop{\uu}\limits_{i=2, j=1}^\infty J_{(1, i)(1, j)} \sci M((a_1, \frac 12)|\ga)$ and
$E(X : X\in \mathop{\uu}\limits_{j=1}^\infty J_{(1,1)(1, j)}\uu\mathop{\uu}\limits_{i=2, j=1}^\infty J_{(1, i)(1, j)})=(\frac{1}{18},\frac{1}{2}).$
So, for any $(x_1, x_2) \in \mathop{\uu}\limits_{i=2, j=1}^\infty J_{(1,1)(i, j)}\uu\mathop{\uu}\limits_{\underset{j=1}{k=1,i=2,}}^\infty J_{(k, 2)(i, j)}$, $\min_{c\in \ga}\|(x_1, x_2)-c\|^2\geq \|(x_1, x_2)-(\frac 16, \frac 12)\|^2$. If
$A=\mathop{\uu}\limits_{j=1}^\infty J_{(1,1)(1, j)}\ \cup \ \mathop{\uu}\limits_{i=2, j=1}^\infty J_{(1, i)(1, j)}, $ $B=\mathop{\uu}\limits_{i=2, j=1}^\infty J_{(1,1)(i, j)} \cup \mathop{\uu}\limits_{\underset{j=1}{k=1,i=2,}}^\infty J_{(k, 2)(i, j)} , $
$ A^{\prime}=\mathop{\uu}\limits_{j=1}^\infty J_{(1,1)(1, j)} $ and $B^{\prime}=\mathop{\uu}\limits_{\underset{j=1}{k=1,i=2,}}^\infty J_{(k, 2)(i, j)}, $ then

\begin{align*}
&\int m(x)^2  dP>\mathop{\int}\limits_{A} \|(x_1, x_2)-(\frac{1}{18},\frac{1}{2})\|^2dP+\mathop{\int}\limits_{B}\|(x_1, x_2)-(\frac 16, \frac 12)\|^2 dP\\
&=2 \mathop{\int}\limits_{A^{\prime}} \|x-(\frac{1}{18},\frac{1}{2})\|^2dP+\mathop{\int}\limits_{\mathop{\uu}\limits_{i=2, j=1}^\infty J_{(1,1)(i, j)}}\|x-(\frac 16, \frac 12)\|^2 dP
 +\mathop{\int}\limits_{B^{\prime}}\|x-(\frac 16, \frac 12)\|^2 dP\\
&=2\cdot\frac{41}{2592}+\frac{5}{288}+\frac{551}{14688}=\frac{953}{11016}>V_3,
\end{align*}
which is a contradiction. Similarly, if we assume $\frac 13 \leq a_2< \frac 12$, a contradiction will arise. Therefore, all the points in $\ga$ can not lie on the line $x_2=\frac 12$. Let $(a_1, b_1)$ and $(a_3, b_3)$ lie on the line $x_2=\frac 12$, and $(a_2, b_2)$ is above or below the horizontal line $x_2=\frac 12$. If $(a_2, b_2)$ is above the horizontal line then the quantization error can be strictly reduced by moving $(a_1, b_1)$ to $A_1$ and $(a_3, b_3)$ to $A_2$ contradicting the fact that $\ga$ is an optimal set. Similarly, if $(a_2, b_2)$ is below the horizontal line a contradiction will arise. All these contradictions arise due to our assumption that $\ga$ does not contain any point from $\mathop{\uu}\limits_{i=1}^4 A_i$. Hence, $\ga$ contains at least one point from $\mathop{\uu}\limits_{i=1}^4 A_i$.
In order to complete the proof of the Proposition, first we will prove the following claim:
%\smallskip

\tbf{Claim.} \tit{$\te{card}(\set{i : \ga \ii A_i \neq \es, \, 1\leq i\leq 4})=2$.} \\
For the sake of contradiction, assume that $\te{card}(\set{i : \ga \ii A_i \neq \es, \, 1\leq i\leq 4})=1$. Then, without any loss of generality we assume that $(a_1, b_1) \in A_1$ and $(a_i, b_i) \not \in A_2\uu A_3 \uu A_4$ for $i=2, 3$. Due to  symmetry of the affine set with respect to the diagonal $x_2=x_1$, we can assume that $(a_1, b_1) \in A_1$ lies on the diagonal $x_2=x_1$; $(a_2, b_2)$ and $(a_3, b_3)$ are equidistant from the diagonal $x_2=x_1$ and are in opposite sides of the diagonal $x_2=x_1$.  Now, consider the following cases:

\tit{Case 1.} Assume that both $(a_2, b_2)$ and $(a_3, b_3)$ are below the diagonal $x_2=1-x_1$, but not in $A_1\uu A_2\uu A_3$. Let $(a_2, b_2)$ be above the diagonal $x_2=x_1$ and $(a_3, b_3)$ be below the diagonal $x_2=x_1$. In that case, quantization error can be strictly reduced by moving $(a_2, b_2)$ to $A_3$ and $(a_3, b_3)$ to $A_2$ which contradicts the optimality of $\ga$.

\tit{Case 2.} Assume that both $(a_2, b_2)$ and $(a_3, b_3)$ are above the diagonal $x_2=1-x_1$. Let $(a_2, b_2)$ lie above the diagonal $x_2=x_1$ and $(a_3, b_3)$ lie below the diagonal $x_2=x_1$. Then, due to symmetry we can assume that $(a_1, b_1)=(\frac 16, \frac 16)$ which is the centroid of $A_1$, $(a_2, b_2)=(\frac 12, \frac 56)$ which is the midpoint of the line segment joining the centroids of $A_3$ and $A_4$,  $(a_3, b_3)=(\frac 56, \frac 12)$ which is the midpoint of the line segment joining the centroids of $A_2$ and $A_4$. Then,

\begin{align*}
&\int m(x)^2 dP= \mathop{\int}\limits_{J_{(1,1)}} m(x)^2 dP+\mathop{\int}\limits_{J_{(1,1)(\es, \infty)}} m(x)^2 dP +\mathop{\int}\limits_{J_{(1,1)(\infty, \es)}}m(x)^2 dP+\mathop{\int}\limits_{J_{(1,1)(\infty, \infty)}}m(x)^2dP\\
&\geq \frac{1}{144}+ \mathop{\int}\limits_{J_{(1,1)(\es, \infty)}}\|x-(\frac 12,\frac 56)\|^2 dP+\mathop{\int}\limits_{J_{(1,1)(\infty, \es)}}\|x-(\frac 56, \frac 12)\|^2 dP +\mathop{\int}\limits_{\mathop{\uu}\limits_{\underset{j=i+1}{i=2}}^\infty J_{(i, j)}}\|x-(\frac 12, \frac 56)\|^2dP\\
&=\frac{1}{144}+\frac{5}{144}+\frac{5}{144}+\frac{1381}{166320}=\frac{7043}{83160}>V_3,
\end{align*}
which is a contradiction. Thus, $\te{card}(\set{i : \ga \ii A_i \neq \es, \, 1\leq i\leq 4})=1$ cannot hold.

Next, for the sake of contradiction, assume that $\te{card}(\set{i : \ga \ii A_i \neq \es, \, 1\leq i\leq 4})=3$. Then, without any loss of generality we assume that $(a_1,b_1) \in  A_3$, $(a_2, b_2) \in A_2$ and $(a_3, b_3) \in A_4$. Let $A_{11}$ and $A_{12}$ be the regions of $A_1$ which are respectively above and below the diagonal of $A_1$ passing through $(0, 0)$. Due to symmetry, we must have
$A_3\uu A_{11} \sci M((a_1, b_1)|\ga)$ and $A_2\uu A_{12}\sci M((a_2, b_2)|\ga)$. Notice that $A_3\uu A_{11} \sci M((a_1, b_1)|\ga)$ implies
\[A_3\uu \mathop{\uu}\limits_{i=1, j=i+1} J_{(1,1)(i, j)}\uu \mathop{\uu}\limits_{\overset{k=1, i=1}{j=i+1}} J_{(1,1)(k,k)(i, j)}\sci M((a_1, b_1)|\ga),\]
and using \eqref{eq00}, we have
\[E(X : X\in A_3\uu \mathop{\uu}\limits_{i=1, j=i+1} J_{(1,1)(i, j)}\uu \mathop{\uu}\limits_{\overset{k=1, i=1}{j=i+1}} J_{(1,1)(k,k)(i, j)}=(\frac{1385}{9438},\frac{6173}{9438}),\]
which shows that the point $(a_1, b_1)$ falls below the line $x_2=\frac 23$, which is a contradiction as we assumed that $(a_1, b_1)\in A_3$. This contradiction arises due to our assumption that $\te{card}(\set{i : \ga \ii A_i \neq \es, \, 1\leq i\leq 4})=3$. Hence, we conclude that $\te{card}(\set{i : \ga \ii A_i \neq \es, \, 1\leq i\leq 4})=2$, which proves the claim.

By the claim, we assume that $(a_1, b_1) \in A_1$ and $(a_3, b_3) \in A_2$. Notice that $A_1, A_2, A_3, A_4$ are geometrically symmetric  as well as their corresponding centroids are symmetrically distributed over the square $[0, 1]\times [0, 1]$. Without any loss of generality, we can assume that the optimal point $(a_1, b_1)$ is the centroid of $A_1$, i.e., $(a_1, b_1)=(\frac 16, \frac 16)$. Then, due to symmetry with respect to the line $x_1=\frac 12$, it follows that $(a_3, b_3)=\te{centroid of } A_2=(\frac 56, \frac 16)$, and $(a_2, b_2)$ lies on $x_1=\frac 12$ but above the line $x_2=\frac 12$. Now, notice that
 \[\mathop{\min}\limits_{(a_3, b_3) \in [\frac 13, \frac 23]\times [\frac23, 1]}\set{\|(\frac 16, \frac 56)-(a_3, b_3)\|^2 +\|(\frac 56, \frac 56)-(a_3, b_3)\|^2}=\frac 2{9},\]
which occurs when $(a_3, b_3)=\te{center of } [\frac 13, \frac 23]\times [\frac23, 1]=(\frac 12, \frac 56)$. Moreover, the three points $(\frac 16, \frac 16)$, $(\frac 56, \frac 16)$ and $(\frac 12, \frac 56)$ are the centroids of their own Voronoi regions. Thus, $\set{(\frac 16, \frac 16), (\frac 56, \frac 16), (\frac 12, \frac 56)}$ forms an optimal set of three-means with quantization error $V_3=\frac 1{12}.$ Hence, the proposition follows.
\end{proof}

\begin{remark}
Due to symmetry, in addition to the optimal set given in Proposition~\ref{prop2}, there are three more optimal sets of three-means with quantization error $V_3=\frac 1{12}$ (see Figure~\ref{Fig3}).
\end{remark}

\section{Affine measures}

In this section, we show that the affine measure $P$ under consideration is the direct product of the Cantor distribution $P_c$.
% with suport the standard Cantor set generated by the similitudes $U_1$ and $U_2$ such that $U_1(x)=\frac 13 x$ and $U_2(x)=\frac 13 x+\frac 23$ for all $x \in \D R$.
%, i.e., $P_c=\frac 12 P_c\circ U_1^{-1}+\frac 12 P_c \circ U_2^{-1}$. Then $P_c$ has support the Cantor set $C$ generated by the similitudes $U_1$ and $U_2$.

For the rest of the article, by a word $\gs$ of length $k$ over the alphabet $\set{1, 2}$, it is meant $\gs:=\gs_1\gs_2\cdots \gs_k\in \set{1, 2}^k$, $k\geq 1$. By a word of length zero it is meant the empty word $\es$. $\set{1, 2}^\ast$ represents the set of all words over the alphabet $\set{1, 2}$ including the empty word $\es$. Length of a word $\gs \in \set{1, 2}^\ast$ is denoted by $|\gs|$. If $\gs=\gs_1\gs_2\cdots \gs_k$, we write $U_\gs:=U_{\gs_1}\circ U_{\gs_2}\circ \cdots \circ U_{\gs_k}$. $U_\es$ represents the identity mapping on $\D R$. By $u_\gs$ we represent the similarity ratio of $U_\gs$.  If $X_c$ is the random variable with distribution $P_c$, then $E(X_c)=\frac 12$ and $V(X_c)=\frac 18$  \cite{GL3}. For $\gs\in \set{1, 2}^\ast$, write $A(\gs):=U_\gs(\frac 12)$. Notice that for $\gs \in \set{1, 2}^\ast$, we have $\frac 12 (A(\gs 1)+A(\gs 2))=A(\gs)$, $u_\gs=\frac{1}{3^{|\gs|}}$, the contractive factor of $U_\gs , $ and for the empty word $\es$, $A(\es)=\frac 12$. For  $\gs \in \set{1, 2}^\ast$ define $A_\gs:=U_\gs[0, 1]$. For any positive integer $n$, by $2^{\ast n}$ it is meant the concatenation of the symbol 2 with itself $n$-times successively, i.e., $2^{\ast n}=222 \cdots (n \te{ times})$, with the convention that $2^{\ast 0}$ is the empty word. For any positive integer $k$, by $\set{1, 2}^{k\ast 2}$ it is meant the direct product of the set $\set{1, 2}^k$ with itself. By $\set{1, 2}^{0\ast2}$ it is meant the set $\set{(\es, \es)}$. Also, recall the notations defined in Section~2. Let us now introduce the map
$F : \D N^\ast \uu \set{(\gs, \infty) : \gs \in \D N^\ast } \to \set{1,2}^\ast$ such that
\begin{equation} \label{eq51}  F(x)=\left\{\begin{array} {ll}
f(\gs_1)f(\gs_2)\cdots f(\gs_{|\gs|}) &\te{ if } x=\gs=\gs_1\gs_2\cdots \gs_{|\gs|},\\
f(\gs_1)f(\gs_2)\cdots f(\gs_{|\gs|}, \infty) &\te{ if } x=(\gs_1\gs_2\cdots \gs_{|\gs|}, \infty),\\
\es  &\te{ if } x=\es,
\end{array}\right.
\end{equation}
where $f : \D N \uu \set{(n, \infty) : n \in \D N } \to \set{1,2}^\ast\setminus \set{\es}$ is such that
\begin{equation*}  f(x)=\left\{\begin{array} {ll}
2^{\ast (n-1)}1 &\te{ if } x=n \te{ for some } n \in \D N,\\
2^{\ast n} & \te{ if } x=(n, \infty) \te{ for some } n \in \D N.
\end{array}\right.
\end{equation*}
The function $f$ is one-to-one and onto, and consequently, $F$ is also one-to-one and onto. For any  $\gs \in \D N^\ast$, write $AF(\gs):=A(F(\gs))$ and $AF(\gs, \infty):=A(F(\gs, \infty))$.

%\begin{remark}In the sequel, we will show that t
The map $F$ is instrumental in converting the infinitely generated affine measure $P$ to a finitely generated affine measure $P_c\times P_c$.  Furthermore, to improve the clarity of the arguments, we will write $T_i$ for $S_{(i, j)}^{(1)}$, and $T_j$ for $S_{(i, j)}^{(2)}$, where $T_k$ for all $k\geq 1$ form an infinite collection of similarity mappings on $\D R$ such that $T_k(x)=\frac 1 {3^k} x+1-\frac 1{3^{k-1}}$ for all $x \in \D R$. Thus, if $\go=(i_1, j_1)(i_2, j_2)\cdots (i_n, j_n)$, then $S_{\go}^{(1)}=T_{i_1}\circ \cdots \circ T_{i_n}=T_{i_1i_2\cdots i_n}$ and $S_{\go}^{(2)}=T_{j_1}\circ \cdots \circ T_{j_n}=T_{j_1j_2\cdots j_n}$ for all $n\geq 1$. Again, $T_\es$ is the identity mapping on $\D R$.
%\end{remark}

\medskip

\begin{lemma} \label{lemma31}
Let $T_k$ for $k\geq 1$ be the infinite collection of similitudes defined above, and $U_1$ and $U_2$ be the similitudes generating the Cantor set. Then, for any $\gs \in \D N^\ast$ and $x \in \D R$, we have
$T_\gs(x)=U_{F(\gs)}(x).$
\end{lemma}
\begin{proof}  If $\gs=1$, then $T_1(x)=\frac 13 x=U_1(x)=U_{F(1)}(x)$ for any $x \in \D R.$  Assume that the lemma is true if $\gs=k$ for some positive integer $k$, i.e., $T_k(x)=U_{F(k)}(x)$.  Then,
%We now show that $T_{k+1}(x)=U_{F(k+1)}(x)$. See that
\begin{align*} &U_{F(k+1)}(x)=U_{2^{\ast k}1}(x)=U_{2^{\ast (k-1)}21}(x)=U_{2^{\ast (k-1)}}U_{21}(x)=U_{2^{\ast (k-1)}}(\frac 19 x+\frac 23)\\
&=U_{2^{\ast (k-1)}1}(3(\frac 19 x+\frac 23))=U_{F(k)}(\frac 13x+2)=T_k(\frac 13x+2)=\frac 1 {3^k} (\frac 13x+2)+1-\frac 1{3^{k-1}}\\
&=\frac 1 {3^{k+1}}x+1-\frac 1{3^{k}}=T_{k+1}(x).
\end{align*}
Thus, by the Principle of Mathematical Induction, $T_k(x)=U_{F(k)}(x)$ for all $k\in \D N$. Again, for any $\gt, \gd \in \D N^\ast$, by \eqref{eq51}, it follows that $F(\gs\gd)=F(\gs)F(\gd)$. Hence, for any $\gs=\gs_1\gs_2\cdots\gs_n\in \D N^\ast$, $n\geq 1$, we have
\[T_\gs(x)=T_{\gs_1}\circ T_{\gs_2}\circ\cdots\circ T_{\gs_n}(x)=U_{F(\gs_1)}\circ U_{F(\gs_2)}\circ \cdots \circ U_{F(\gs_n)}(x)=U_{F(\gs)}(x),\]
which completes the proof.
\end{proof}

%With the help of the above lemma we prove the following lemma.
\begin{lemma} \label{lemma32} Let $\go\in \mathcal{I}^\ast$, and $F$ be the function as defined in \eqref{eq51}. Then for $r=1, 2$, we have
$
 AF(\go^{(r)})=S_\go^{(r)}(\frac 12)$, and  $AF(\go^{(r)}, \infty)=S_{\go^-(\go^{(1)}_{|\go|}+1,\,  \go^{(2)}_{|\go|}+1)}^{(r)}(\frac 12)+s^{(r)}_{\go^-(\go^{(1)}_{|\go|}+1,\,  \go^{(2)}_{|\go|}+1)}$.
\end{lemma}

\begin{proof} By Lemma~\ref{lemma31}, we have
\begin{equation*} AF(\go^{(1)})=U_{F(\go^{(1)})}(\frac 12)=T_{\go^{(1)}}(\frac 12)=S_\go^{(1)}(\frac 12), \te{ and similarly } AF(\go^{(2)})=S_\go^{(2)}(\frac 12).\end{equation*}
Without any loss of generality, we can assume  $\go=(i_1, j_1)(i_2, j_2)\cdots(i_k, j_k)$ for $k\geq 1$. Then,
\begin{align*}
&AF(\go^{(1)}, \infty)=U_{F(i_1i_2\cdots i_k, \infty)}(\frac 12)=U_{F(i_1i_2\cdots i_{k-1})}\circ U_{F(i_k, \infty)}(\frac 12)=U_{F(i_1i_2\cdots i_{k-1})}\circ U_{2^{\ast i_k}}(\frac 12)\\
&=U_{F(i_1i_2\cdots i_{k-1})}\circ U_{2^{\ast i_k}1}(U_1^{-1}(\frac 12))=U_{F(i_1i_2\cdots i_{k-1})}\circ U_{F(i_{k}+1)}(\frac 32)=U_{F(i_1i_2\cdots i_{k-1}(i_{k}+1))}(\frac 32)\\
&=T_{i_1i_2\cdots i_{k-1}(i_{k}+1)}(\frac 32)=S_{\go^-(i_k+1,\, j_{k}+1)}^{(1)}(\frac 32).
\end{align*}
Since, $S_{(i_k+1,\, j_{k}+1)}^{(1)}(\frac 32)-S_{(i_k+1,\, j_{k}+1)}^{(1)}(\frac 12)=\frac 1{3^{i_k+1}} \frac 32+1-\frac 1{3^{i_k}}-\frac 1{3^{i_k+1}} \frac 12-1+\frac 1{3^{i_k}}=\frac 1{3^{i_k+1}},$ we have
\begin{align*}
&S_{\go^-(i_k+1,\, j_{k}+1)}^{(1)}(\frac 32)-S_{\go^-(i_k+1,\, j_{k}+1)}^{(1)}(\frac 12)= s_{\go^-}^{(1)}(S_{(i_k+1,\, j_{k}+1)}^{(1)}(\frac 32)-S_{(i_k+1,\, j_{k}+1)}^{(1)}(\frac 12))=s_{\go^-}^{(1)}\frac 1{3^{i_k+1}}\\
&=s_{\go^-(i_k+1, j_k+1)}^{(1)}=s^{(1)}_{\go^-(\go^{(1)}_{|\go|}+1,\,  \go^{(2)}_{|\go|}+1)}, \ \text{which yields}
\end{align*}
%yielding
$AF(\go^{(1)}, \infty)=S_{\go^-(i_k+1,\, j_{k}+1)}^{(1)}(\frac 32)=S_{\go^-(\go^{(1)}_{|\go|}+1,\,  \go^{(2)}_{|\go|}+1)}^{(1)}(\frac 12)+s^{(1)}_{\go^-(\go^{(1)}_{|\go|}+1,\,  \go^{(2)}_{|\go|}+1)}.$  Similarly,
$AF(\go^{(2)}, \infty)=S_{\go^-(\go^{(1)}_{|\go|}+1,\,  \go^{(2)}_{|\go|}+1)}^{(2)}(\frac 12)+s^{(2)}_{\go^-(\go^{(1)}_{|\go|}+1,\,  \go^{(2)}_{|\go|}+1)}$.
%which completes the proof.
\end{proof}

\begin{remark}
By Lemma~\ref{lemma4} and  Lemma~\ref{lemma32}, for any $\go \in \mathcal{I}^\ast$, we have
\begin{align*}
& a(\go)=(AF(\go^{(1)}), AF(\go^{(2)})),  \ \
 a(\go(\infty, \infty))=(AF(\go^{(1)}, \infty), AF(\go^{(2)}, \infty)),\\
& a(\go(\infty, \es))=(AF(\go^{(1)}, \infty), AF(\go^{(2)})), \ \ \text{and} \ \  a(\go(\es, \infty))=(AF(\go^{(1)}), AF(\go^{(2)}, \infty)).
\end{align*}
\end{remark}
The following example illustrates the outcome of the lemma above.

\begin{example} \label{ex1}
$a((1,1))=(AF(1), AF(1))=(A(1), A(1))=(\frac 16, \frac 16)$, \\
$a((1,1)(\infty, \es))=(AF(1, \infty), AF(1))=(A(2), A(1))=(\frac 56, \frac 16)$, \\
$a((1,1)(\es, \infty))=(AF(1), AF(1, \infty))=(A(1), A(2))=(\frac 16, \frac 56)$,\\
$a((1,1)(\infty, \infty))=(AF(1, \infty), AF(1, \infty))=(A(2), A(2))=(\frac 56, \frac 56)$, \\
 $a((1,1)(1,1))=(AF(11), AF(11))=(A(11), A(11))=(\frac 1{18}, \frac 1{18})$, \\
 $a((1,1)(1,1)(\infty, \es))=(AF(11, \infty), AF(11))=(A(12), A(11))=(\frac 5{18}, \frac 1{18})$,\\ $a((1,1)(1,1)(\es, \infty))=(AF(11), AF(11, \infty))=(A(11), A(12))=(\frac 1{18}, \frac 5{18})$, and\\ $a((1,1)(1,1)(\infty, \infty))=(AF(11, \infty), AF(11, \infty))=(A(12), A(12))=(\frac 5{18}, \frac 5{18})$, etc.
\end{example}

\medskip

\begin{lemma}\label{lemma33}
Let $\mu=\sum_{k=1}^\infty \frac 1{2^k} \mu\circ T_k^{-1}$. Then, for any $\gs\in \D N^\ast$, we have $\mu(T_\gs[0, 1])=P_c(A_{F(\gs)})$, where $P_c:=\frac 1 2 P_c\circ U_1^{-1} +\frac 1 2 P_c\circ U_2^{-1}$.
\end{lemma}

\begin{proof} Without any loss of generality, let $\gs=i_1 i_2\cdots i_k$ for any $k\geq 1$. See that $F(\gs)=F(i_1)F(i_2)\cdots F(i_k)$, and thus $|F(\gs)|=|F(i_1)|+|F(i_2)|+\cdots+|F(i_k)|=i_1+i_2+\cdots+i_k . $  Consequently,
\[\mu(T_\gs[0, 1])=\frac 1{2^{i_1+i_2+\cdots+i_k}}=\frac 1{2^{|F(\gs)|}}=P_c(A_{F(\gs)}),\]
which proves the lemma.
\end{proof}

%The following proposition plays an important role in the paper.

\begin{prop}\label{prop10} {\it Let $P$ be the affine measure. Then, $P=P_c\times P_c$, where $P_c$ is the Cantor distribution.}
\end{prop}
\begin{proof} Borel $\gs$-algebra on the affine set is generated by all sets of the form $J_{(\gd, \gt)}$ for $(\gd, \gt)\in \mathcal{I}^\ast$, where $J_{(\gd,\gt)}=S_{(\gd,\gt)}([0, 1]\times [0,1])$. Notice that
\[J_{(\gd,\gt)}=T_{\gd}[0, 1]\times T_{\gt}[0, 1]=U_{F(\gd)}[0, 1]\times U_{F(\gt)}[0, 1]=A_{F(\gd)}\times A_{F(\gt)}.\]
Again, the sets of the form $A_\ga$, where $\ga \in\set{1, 2}^\ast$, generate the Borel $\gs$-algebra on the Cantor set $C$. Thus, we see that the Borel $\gs$-algebra of the affine set is the same as the product of the Borel $\gs$-algebras on the Cantor set. Moreover, for any $(\gd, \gt) \in \mathcal{I}^\ast$, by Remark~\ref{rem1} and Lemma~\ref{lemma33}, we have
\[P(J_{(\gd, \gt)})=\mu(T_{\gd}[0, 1]) \mu(T_{\gt}[0, 1])=P_c(A_{F(\gd)})P_c(A_{F(\gt)})=(P_c\times P_c)(A_{F(\gd)}\times A_{F(\gt)}).\]
Hence, the proposition follows.
\end{proof}

\begin{remark} By Proposition~\ref{prop10}, it follows that the optimal sets of $n$-means for $P$ are the same as the optimal sets $n$-means for the product measure $P_c\times P_c$ on the affine set. Moreover, for $k\geq 1$ we can write
\[P=P_c\times P_c=\sum_{(\gs, \gt) \in \set{1, 2}^{k\ast 2}}\frac 1{4^k} (P_c\times P_c)\circ (U_\gs, U_\gt)^{-1}, \]
where for $(x_1, x_2) \in \D R^2$, $(U_\gs, U_\gt)^{-1}(x_1, x_2)=(U_\gs^{-1}(x_1), U_\gt^{-1}(x_2))$.
\end{remark}

\section{Optimal sets of $n$-means for all $n\geq 4$}

In this section we will derive closed formulas to determine the optimal sets of $n$-means and the $n$th quantization error for all $n\geq 4$.
For $(\gs, \gt) \in \set{1, 2}^{k\ast 2}$, write $A_{(\gs,\gt)}:=A_\gs\times A_\gt$ and $U_{(\gs, \gt)}:=(U_\gs, U_\gt)$.

%The following lemma plays an important role in the rest of the paper.
\begin{lemma} \label{lemma41}
Let $\ga$ be an optimal set of $n$-means with $n\geq 4$. Then, $\ga\ii  A_{(i, j)}\neq \es$ for all $1\leq i, j\leq 2$.

\end{lemma}

\begin{proof} Let $\ga$ be an optimal set of $n$-means for $n\geq 4$.
As the optimal points are the centroids of their own Voronoi regions we have $\ga \sci A_\es\times A_\es:=[0, 1]\times [0, 1]$.

Consider the four-point set $\gb$ given by $\gb=\set{(A(i), A(j)) : 1\leq i, j\leq 2}$. Then,
\begin{align*}
\int \min_{c\in \gb}\|x-c\|^2 dP=\sum_{i, j=1}^2 \int_{A_{(i,j)}}\|x- (A(i), A(j))\|^2 d(P_c\times P_c)=\sum_{i, j=1}^2\frac 14(\frac 19+\frac 19)\frac 18=\frac 1{36}.
\end{align*}
Since $V_4$ is the quantization error of four-means, we have $\frac 1{36} \geq V_4\geq V_n$.

Assume that $\ga$ does not contain any point from $\mathop{\uu}\limits_{i, j=1}^2A_{(i,j)}$. We know that
\begin{equation} \label{eq100}
\sum_{(a, b) \in \ga} (a, b) P(M(a, b)|\ga))=(\frac 12, \frac 12).
\end{equation}
If all the points of $\ga$ are below the line $x_2=\frac 12$, i.e., if $b<\frac 12$ then by \eqref{eq100}, we see that $\frac 12=\sum_{(a, b) \in \ga} b P(M(a, b)|\ga))<\sum_{(a, b) \in \ga} \frac 12 P(M(a, b)|\ga))=\frac 12$, which is a contradiction. Similarly, it follows that if all the points of $\ga$ are above the line $x_2=\frac 12$, or left of the line $x_1=\frac 12$, or right of the line $x_1=\frac 12$, a contradiction will arise.

Next, suppose that all the points of $\ga$ are on the line $x_2=\frac 12 $.  We will consider two cases: $n=4$ and $n>4.$  When $n=4 , $ let $\ga=\set{(a_i, \frac 12) : 1\leq i\leq 4}$ with $a_i<a_j$ for $i<j$. Due to symmetry, we can assume that the boundary of the Voronoi regions of the points $(a_1, \frac 12)$, $(a_2, \frac 12)$, $(a_3, \frac 12)$, and $(a_4, \frac 12)$ are respectively $x_1=\frac 16$, $x_1=\frac 12$, and $x_1=\frac 56$ yielding $\ga=\set{(\frac 1{18}, \frac 12), (\frac 5{18}, \frac 12), (\frac {13}{18}, \frac 12), (\frac {17}{18}, \frac 12)}$, and then writing $B:=A_{(11, 11)}\uu A_{(11, 12)}\uu A_{(11, 21)}\uu A_{(11,22)}$, by symmetry we have
\begin{align*}
&\int \min_{c\in \ga}\|x-c\|^2 dP=4 \mathop{\int}\limits_{B}\|x-(\frac1{ 18}, \frac 12)\|^2d(P_c\times P_c)\\
&=8 \mathop{\int}\limits_{A_{(11,11)}}\|x-(\frac1{ 18}, \frac 12)\|^2d(P_c\times P_c)+ 8 \mathop{\int}\limits_{A_{(11,12)}}\|x-(\frac1{ 18}, \frac 12)\|^2d(P_c\times P_c)\\
&=8(\frac{65}{5184}+\frac{17}{5184})=\frac{41}{324}>V_4,
\end{align*}
which is a contradiction. We consider the case $n>4$.
Since for any $(x_1, x_2)\in \mathop{\uu}\limits_{i, j=1}^2A_{ij}$,  $\min_{c\in\ga}\|(x_1, x_2)-c\|^2 \geq \frac 1{36}$, we have
\begin{align*}
&\int \min_{c\in \ga}\|x-c\|^2 dP=\sum_{i, j=1}^2 \int_{A_{(i, j)}}\min_{c\in \ga}\|x-c\|^2d(P_c\times P_c)\geq \sum_{i, j=1}^2\int_{A_{(i, j)}}\frac 1 {36} d(P_c\times P_c)=\frac 1{36},
\end{align*}
which implies $\frac 1 {36}\geq V_4>V_n$, a contradiction. Thus, we see that all the points of $\ga$ can not lie on $x_2=\frac12$. Similarly, all the points of $\ga$ can not lie on $x_1=\frac 12$.

Notice that the lines $x_1=\frac 12$ and $x_2=\frac 12$ partition the square $[0, 1]\times [0, 1]$ into four quadrants with center $(\frac 12, \frac 12)$.
If $n=4k$ for some positive integer $k$, due to symmetry, we can assume that each quadrant contains $k$-points from the set $\ga$. But then, any of the $k$ points in the quadrant containing a basic rectangle $A_{(i, j)}$ can be moved to $A_{(i, j)}$ which strictly reduce the quantization error, and it gives a contradiction as we assumed that the set $\ga$ is an optimal set of $n$-means and $\ga$ does not contain any point from $A_{(i, j)}$ for $1\leq i, j\leq 2$.

If $n=4k+1, 4k+2$, or  $n=4k+3$, then, again due to symmetry, each quadrant gets at least $k$ points. Then, as in the case $n=4k,$ here also, one can strictly reduce the quantization error by moving a point in the quadrant containing a basic rectangle $A_{(i, j)}$ to $A_{(i, j)}$ for $1\leq i, j\leq 2$, which is a contradiction.

Thus, we have proved that $\ga\ii  A_{(i, j)}\neq \es$ for all $1\leq i, j\leq 2$.
\end{proof}

\begin{lemma} \label{lemma42} Let $\ga$ be an optimal set of $n$-means with $n\geq 4$. Then, $\ga \sci \mathop{\uu}\limits_{i, j=1}^2 A_{(i,j)}$.

\end{lemma}

\begin{proof}  By Lemma 5.1, we know that $\ga\ii  A_{(i, j)}\neq \es$ for all $1\leq i, j\leq 2$.  Now, we will prove the statement by considering four distinct cases:

\tit{Case 1: $n=4 k$ for some integer $k\geq 1$.}

In this case, due to symmetry, we can assume that $\ga$ contains $k$ points from each of $A_{(i, j)}$, otherwise, quantization error can be reduced by redistributing the points of $\ga$ equally among $A_{(i, j)}$ for $1\leq i, j\leq 2$, and so $\ga \sci \mathop{\uu}\limits_{i, j=1}^2 A_{(i,j)}$.

\tit{Case 2: $n=4 k+1$ for some integer $k\geq 1$. }

In this case, again due to symmetry, we can assume that $\ga$ contains $k$ points from each of $A_{(i,j)}, $ and if possible, one point, say $(a, b)$, from $A_{(\es, \es)}\setminus \mathop{\uu}\limits_{i, j=1}^2 A_{(i,j)}$.  By symmetry, one can assume that $(a,b)$ is the midpoint of the line segment joining any two centroids of the basic rectangles $A_{(i, j)}$ for $1\leq i, j\leq 2$. Let us first take $(a, b)=(\frac 12, \frac 12)$ which is the center of the affine set. For simplicity, we first assume $k=1$, i.e., $n=5$. Then, $\ga$ contains only one point from each of $A_{(i, j)}$. Let $(a_1, b_1)$ be the point that $\ga$ takes from $A_{(1,1)}$. As $(\frac 12, \frac 12)$ lies on the diagonal $x_2=x_1$, due to symmetry we can also assume that $(a_1, b_1)$ lies on the diagonal $x_2=x_1$.
By Proposition~\ref{prop000}, we have $P(M((\frac 12, \frac 12)|\ga))>0$. This yields that $\frac 12 ((a_1, b_1)+(\frac 12, \frac 12))<(\frac 13, \frac 13)$ which implies $a_1<\frac 16$ and $b_1<\frac 16$. Then, we see that
\[\frac 1{36}=V_4\approx V_5=4 \mathop{\int}\limits_{A_{(1,1)}}\min_{c\in\set{(a_1, b_1), (\frac 12, \frac 12)}}\|x-c\|^2dP>\int\min_{c\in \gb}\|x-c\|^2 dP=\frac{2}{81}\geq V_5,\]
where $\gb=\set{(\frac 1{18}, \frac 1{18}), (\frac 1{18}, \frac 5{18}), (\frac 5{6}, \frac 1{6}), (\frac 1{6}, \frac 5{6}), (\frac 5{6}, \frac 5{6})}$,
which is a contradiction. Similarly, if we take $(a, b)$ as the midpoint of a line segments joining the centroids of any two adjacent basic rectangles $A_{(i, j)}$ for $1\leq i, j\leq 2$, contradiction arises. Proceeding in the similar way, by taking $k=2, 3, \cdots$, we see that contradiction arises at each value $k$ takes.  Therefore, $\ga \sci \mathop{\uu}\limits_{i, j=1}^2 A_{(i,j)}$.

\tit{Case 3: $n=4 k+2$ for some integer $k\geq 1$.}

In this case, due to symmetry, we can assume that $\ga$ contains $k$ points from each of $A_{(i,j)}$, and if possible, two points, say $(a_1, b_1)$ and $(a_2, b_2)$, from $A_{(\es, \es)}\setminus \mathop{\uu}\limits_{i, j=1}^2 A_{(i,j)}$. Then, by symmetry, we can assume that $(a_1, b_1)$ lies on the midpoint of the line segment joining the centroids of $A_{(1,1)}$,  $A_{(2,1)}$;  and $(a_2, b_2)$ lies on the midpoint of the line segment joining the centroids of $A_{(1,2)}$ and $A_{(2,2)}$. As in Case 2, this leads to a contradiction. Thus, $\ga \sci \mathop{\uu}\limits_{i, j=1}^2 A_{(i,j)}$.

\tit{Case 4: $n=4 k+3$ for some integer $k\geq 1$.}
Due to symmetry, in this case, we can assume that each of $A_{(1,1)}$ and $A_{(2,1)}$ gets $k+1$ points; each of $A_{(1,2)}$ and $A_{(2,2)}$ gets $k$ points. The remaining one point lies on the midpoint of the line segment joining the centroids of $A_{(1,2)}$ and $A_{(2,2)}$. But, in that case, proceeding as in Case 2, we can show that a contradiction arises. Thus, $\ga \sci \mathop{\uu}\limits_{i, j=1}^2 A_{(i,j)}$.

We have shown that in all possible cases $\ga \sci \mathop{\uu}\limits_{i, j=1}^2 A_{(i,j)}$; hence, the lemma follows.
\end{proof}

%The following corollary follows from Lemma 5.2.

\begin{cor} \label{cor21} {\it  The set $\set{(\frac 16, \frac 16), (\frac 56, \frac 16), (\frac 16, \frac 56), (\frac 56, \frac 56)}$ is a unique optimal set of four-means of the affine measure $P$ with quantization error $V_4=\frac 1{36}.$}
\end{cor}

\begin{remark} Let $\ga$ be an optimal set of $n$-means, and $n_{ij}=\te{card}(\gb_{ij})$ where $\gb_{ij}=\ga\ii A_{(i, j)}$ for $1\leq i, j\leq 2$. Then,
$0\leq |n_{ij}-n_{pq}|\leq 1$ for $1\leq i, j, p, q\leq 2$.

\end{remark}

\begin{lemma} \label{lemma22} Let $n\geq 4$ and $\ga$ be an optimal set of $n$-means for the product measure $P_c\times P_c$. For $1\leq i, j\leq 2$, set $\gb_{ij}:=\ga\ii A_{(i, j)}$, and let $n_{ij}=\te{card}(\gb_{ij})$. Then, $U_{(i, j)}^{-1}(\gb_{ij})$ is an optimal set of $n_{ij}$-means, and $V_n=\mathop{\sum}\limits_{i, j=1}^2 \frac 1{36} V_{n_{ij}}$.

\end{lemma}
\begin{proof} For $n\geq 4$, by Lemma~\ref{lemma41}, we have $\ga=\mathop{\uu}\limits_{i, j=1}^2\gb_{ij}$, $n=\mathop{\sum}\limits_{i,j=1}^2 n_{ij}$, and so
\[
V_n=\mathop{\sum}\limits_{i, j=1}^2\mathop{\int}\limits_{A_{(i, j)}}\mathop{\min}\limits_{a\in \gb_{ij}}\|x-a\|^2 d(P_c\times P_c).\]
If $U_{(1, 1)}^{-1}(\gb_{11})$ is not an optimal set of $n_{11}$-means for $P_c\times P_c$, then there exists a set $\gg_{11}\sci \D R^2$ with $\te{card}(\gg_{11})=n_{11}$ such that
$\int \min_{a\in \gg_{11}} \|x-a\|^2 d(P_c\times P_c)<\int \min_{a\in U_{(1,1)}^{-1}(\gb_{11})} \|x-a\|^2 d(P_c\times P_c)$. But then, $\gd:=U_{(1,1)}(\gg_{11})\uu \gb_{12}\uu \gb_{21}\uu \gb_{22}$ is a set of cardinality $n$ and it satisfies $\int\min_{a\in \gd}\|x-a\|^2 d(P_c\times P_c)<\int\min_{a\in \ga}\|x-a\|^2 d(P_c\times P_c), $ contradicting the fact that $\ga$ is an optimal set of $n$-means for $P_c\times P_c$. Similarly, it can be proved that $U_{(1,2)}^{-1}(\gb_{12})$, $U_{(2,1)}^{-1}(\gb_{21})$, and $U_{(2,2)}^{-1}(\gb_{22})$ are optimal sets of $n_{12}$-, $n_{21}$-, and $n_{22}$-means respectively. Thus,
%\begin{align*} &
$$V_n=\mathop{\sum}\limits_{i, j=1}^2\frac 1 4\mathop{\int}\mathop{\min}\limits_{a\in \gb_{ij}}\|x-a\|^2 d((P_c\times P_c)\circ U_{(i,j)}^{-1})=\mathop{\sum}\limits_{i, j=1}^2\frac 1{36}\mathop{\int}\mathop{\min}\limits_{a\in U_{(i,j)}^{-1}(\gb_{ij})}\|x-a\|^2 dP
=\mathop{\sum}\limits_{i, j=1}^2\frac 1 {36}V_{n_{ij}},$$
%\end{align*}
which gives the lemma.
\end{proof}

\begin{prop} \label{prop31}
{\it Let $n\in \D N$ be such that $n=4^{\ell(n)}$ for some positive integer $\ell(n)$. Then, the set
\[\ga_{4^{\ell(n)}}:=\mathop{\uu}\limits_{(\gs, \gt) \in \set{1, 2}^{\ell(n)\ast 2}}\set{(A(\gs), A(\gt))}\] forms a unique optimal set of $n$-means for the affine measure $P$ with quantization error\\
$V_{4^{\ell(n)}}=\frac 1 4 \frac1{9^{\ell(n)}}.$}
\end{prop}

\begin{proof} We will prove the statement by induction. By Corollary~\ref{cor21}, it is true if $\ell(n)=1$. Let us assume that it is true for $n=4^k$ for some positive integer $k$. We now show that it is also true if $n=4^{k+1}$. Let $\gb$ be an optimal set of $4^{k+1}$-means. Set $\gb_{ij}:=\gb \ii A_{(i,j)}$ for $1\leq i, j\leq 2$. Then, by Lemma~\ref{lemma41} and Lemma~\ref{lemma22}, $U_{(i, j)}^{-1}(\gb_{ij})$ is an optimal set of $4^k$-means, and so
$U_{(i, j)}^{-1}(\gb_{ij})=\set{(A(\gs), A(\gt)) : (\gs,\gt) \in \set{1, 2}^{k\ast 2}}$ which implies $\gb_{ij}=\set{(A(i\gs), A(j\gt)) : (\gs,\gt) \in \set{1, 2}^{k\ast 2}}$. Thus, $\gb=\uu_{i, j=1}^2 \gb_{ij}=\set{(A(\gs), A(\gt)) : (\gs, \gt) \in \set{1, 2}^{(k+1)\ast 2}}$ is an optimal set of $4^{k+1}$-means. Since $(A(\gs), A(\gt))$ is the centroid of $A_{(\gs, \gt)}$ for each $(\gs, \gt)\in \mathcal{I}^{k+1}$, the set $\gb$ is unique. Now, by Lemma~\ref{lemma22}, we have the quantization error as
\[V_{k+1}=\mathop{\sum}\limits_{i, j=1}^2 \frac 1{36}V_{^k}=\frac 1 9 \cdot \frac 1{4}\cdot \frac1{9^k}=\frac 14 \frac{1}{9^{k+1}}. \]
Thus, by induction, the proof of the proposition is complete.
\end{proof}

\begin{defi}  \label{defi1} For $n\in \D N$ with $n\geq 4$ let $\ell(n)$ be the unique natural number with $4^{\ell(n)} < n\leq 2\cdot 4^{\ell(n)}$. For $I\sci \set{1, 2}^{\ell(n)\ast 2}$ with card$(I)=n-4^{\ell(n)}$ let $\ga_n(I)$ be the set defined as follows:
\begin{align*} \ga_n(I)&=\mathop{\uu}\limits_{(\gs, \gt) \in \set{1, 2}^{\ell(n)\ast 2} \setminus I}\set{(A(\gs), A(\gt))} \uu (\mathop{\uu}\limits_{ (\gs, \gt) \in I} \set{(A(\gs 1), A(\gt)), (A(\gs 2), A(\gt))}).
\end{align*}
\end{defi}
\begin{remark} In Definition~\ref{defi1}, instead of choosing the set $\set{(A(\gs 1), A(\gt)), (A(\gs 2), A(\gt))}$, one can choose  $\set{(A(\gs), A(\gt 1)), (A(\gs), A(\gt 2))}$, i.e., the set associated with each $(\gs, \gt)\in I$ can be chosen in two different ways. Moreover, the subset $I$ can be chosen from $\set{1, 2}^{\ell(n)\ast 2}$ in ${}^{4^{\ell(n)}}C_{n-4^{\ell(n)}}$ ways. Hence, the number of the sets $\ga_n(I)$ is $2^{\te{card}(I)}\cdot{}^{4^{\ell(n)}}C_{n-4^{\ell(n)}}$.
%If $I$ is an empty set, then $\ga_n(I)=\set{(A(\gs), A(\gt)) : (\gs, \gt) \in \set{1, 2}^{\ell(n)\ast 2}}$.
\end{remark}

The following example illustrates Definition~\ref{defi1}.

\begin{example} Let $n=5$. Then, $\ell(n)=1$, $I\sci \set{1, 2}^{\ast 2}$ with $\te{card}(I)=1$, and so
\begin{align*}
\ga_5(\set{(1,1)})&=\set{(A(1), A(2)), (A(2), A(1)), (A(2), A(2))}\uu\set{(A(11), A(1)), (A(12), A(1))}\\
&=\{(\frac 16, \frac 56), (\frac 56, \frac 16), (\frac 56, \frac 56)\}\uu \set{(\frac 1{18}, \frac 1 6), (\frac 5{18}, \frac 16)}, \end{align*}
or,
\begin{align*} \ga_5(\set{(1,1)})&=\set{(A(1), A(2)), (A(2), A(1)), (A(2), A(2))}\uu\set{(A(1), A(11)), (A(1), A(12))}\\
&=\{(\frac 16, \frac 56), (\frac 56, \frac 16), (\frac 56, \frac 56)\}\uu \set{(\frac 1{6}, \frac 1 {18}), (\frac 1{6}, \frac 5{18})}.
\end{align*}
Similarly, one can get six more sets by taking $I=\set{(1, 2)}$, $\set{(2, 1)}$, or $\set{(2, 2)}$, i.e., the number of the sets $\ga_n(I)$ in this case is $2^{\te{card}(I)}\cdot{}^{4^{\ell(n)}}C_{n-4^{\ell(n)}}=8$.
\end{example}

\begin{prop}\label{prop4}
{\it Let $n\geq 4$ and $\ga_n(I)$ be the set as defined in Definition~\ref{defi1}. Then, $\ga_n(I)$ forms an optimal set of $n$-means with quantization error
\[V_n=\frac 14 \frac 1{36^{\ell(n)}} \Big(2 \cdot 4^{\ell(n)}-n +\frac 5 9(n-4^{\ell(n)})\Big).\]}
\end{prop}
\begin{proof} We have $n=4^{\ell(n)}+k$ where $1\leq k\leq 4^{\ell(n)}$. Set $\gb_{ij}=\ga\ii A_{ij}$ with $n_{ij}=\te{card}(\gb_{ij})$ for $1\leq i, j\leq 2$. Let us prove it by induction. We first assume $k=1$. By Lemma~\ref{lemma41} and Lemma~\ref{lemma22}, we can assume that each of $U_{(i, j)}^{-1}(\gb_{ij})$ for $i=2$ and $j=1, 2$, are optimal sets of $4^{\ell(n)-1}$-means and $U_{(1,1)}^{-1}(\gb_{11})$ is an optimal set of $(4^{\ell(n)-1}+1)$-means. Thus, for $i=2$ and $j=1, 2$, we can write
\begin{align*} U_{(i, j)}^{-1}(\gb_{ij})&=\set{(A(\gs), A(\gt)) : (\gs, \gt) \in \set{1,2}^{(\ell(n)-1)\ast 2}}, \te{ and }\\
U_{(1, 1)}^{-1}(\gb_{11})&=\set{(A(\gs), A(\gt)) : (\gs, \gt) \in \set{1,2}^{(\ell(n)-1)\ast 2}\setminus \set{\gt}}\uu U_\gt(\ga_2),
\end{align*}
for some $\gt \in \set{1,2}^{(\ell(n)-1)\ast 2}$, where $\ga_2$ is an optimal set of two-means. Thus,
\[\ga_n(\set{(1,1)\gt})=\mathop{\uu}\limits_{i, j=1}^2 \gb_{ij}=\set{(A(\gs), A(\gt)) : (\gs, \gt) \in \set{1,2}^{\ell(n)\ast 2}\setminus \set{(1,1)\gt}}\uu U_{(1,1)\gt}(\ga_2),\]
for some $\gt \in \set{1,2}^{(\ell(n)-1)\ast 2}$, where $\ga_2$ is an optimal set of two-means. Notice that instead of choosing $U_{(1,1)}^{-1}(\gb_{11})$ as an optimal set of $(4^{\ell(n)-1}+1)$-means, one can choose any one from $U_{(i, j)}^{-1}(\gb_{ij})$ for $i=2$, $j=1, 2$, as an optimal set of $(4^{\ell(n)-1}+1)$-means. Hence, for $n=4^{\ell(n)}+1$, one can write
\[\ga_n(I)=\mathop{\uu}\limits_{i, j=1}^2 \gb_{ij}=\set{(A(\gs), A(\gt)) : (\gs, \gt) \in \set{1,2}^{\ell(n)\ast 2}\setminus \set{\gt}}\uu U_{\gt}(\ga_2),\]
where $I=\set{\gt}$ for some $\gt \in \set{1,2}^{\ell(n)\ast 2}$ as an optimal set of $n$-means. Thus, we see that the proposition is true if $n=4^{\ell(n)}+k$. Similarly, one can prove that the proposition is true for any $1\leq k\leq 4^{\ell(n)}$.
Then, the quantization error is
\begin{align*}
&V_n=\min_{(a, b)  \in \ga_n(I)}\|x-(a, b)\|^2 dP= \sum_{(\gs, \gt) \in \set{1, 2}^{\ell(n)\ast 2} \setminus I}\int_{A_\gs\times A_\gt}\|x-(A(\gs), A(\gt))\|^2 d(P_c\times P_c) \\
&\qquad \qquad +\sum_{(\gs, \gt) \in  I}\sum_{i=1}^2 \int_{A_{\gs i}\times A_\gt}\|x-(A(\gs i), A(\gt))\|^2 d(P_c\times P_c) \\
&=\sum_{(\gs, \gt) \in \set{1, 2}^{\ell(n)\ast 2} \setminus I} \frac 1{4^{\ell(n)}} (u_{\gs}^2+u_{\gt}^2)\frac 1 8+\sum_{(\gs, \gt) \in  I}\sum_{i=1}^2  \frac 1{4^{\ell(n)}} \frac 12 (u_{\gs i}^2+u_{\gt}^2)\frac 1 8\\
&=\sum_{(\gs, \gt) \in \set{1, 2}^{\ell(n)\ast 2} \setminus I} \frac 1{4^{\ell(n)}} (u_{\gs}^2+u_{\gt}^2)\frac 1 8+\sum_{(\gs, \gt) \in  I} \frac 1{4^{\ell(n)}} (\frac 19 u_{\gs}^2+u_{\gt}^2)\frac 1 8.
\end{align*}
Since, $\te{card}(\set{1, 2}^{\ell(n)\ast 2} \setminus I)=2\cdot 4^{\ell(n)}-n$, $\te{card}(I)=n-4^{\ell(n)}$, $u_\gs=u_\gt=\frac 1{3^{\ell(n)}}$, upon simplification, we have $V_n=\frac 14 \frac 1{36^{\ell(n)}} \Big(2 \cdot 4^{\ell(n)}-n +\frac 5 9(n-4^{\ell(n)})\Big)$.
Thus, the proof of the proposition is complete.
\end{proof}

\begin{defi}  \label{defi3} For $n\in \D N$ with $n\geq 4$ let $\ell(n)$ be the unique natural number with $2\cdot 4^{\ell(n)} < n<  4^{\ell(n)+1}$. For $I\sci \set{1, 2}^{\ell(n)\ast 2}$ with card$(I)=n-2\cdot4^{\ell(n)}$ let $\ga_n(I)$ be the set defined as follows:
\begin{align*} \ga_n(I)&=\mathop{\uu}\limits_{(\gs, \gt) \in \set{1, 2}^{\ell(n)\ast 2} \setminus I}\set{(A(\gs 1), A(\gt)), (A(\gs 2), A(\gt))}\\
 &\uu (\mathop{\uu}\limits_{ (\gs, \gt) \in I} \set{(A(\gs 1), A(\gt 1)), (A(\gs 1), A(\gt2)),  (A(\gs 2), A(\gt))}).
\end{align*}
\end{defi}

\begin{remark} \label{rem10} In Definition~\ref{defi3}, instead of choosing the set $\set{(A(\gs 1), A(\gt)), (A(\gs 2), A(\gt))}$, one can choose  $\set{(A(\gs), A(\gt 1)), (A(\gs), A(\gt 2))}$. Instead of choosing the set \\
$\set{(A(\gs 1), A(\gt 1)), (A(\gs 1), A(\gt2)),  (A(\gs 2), A(\gt))}$, one can choose either the set \\$\set{(A(\gs 1), A(\gt)), (A(\gs 2), A(\gt1)),  (A(\gs 2), A(\gt2))}$, or \\$\set{(A(\gs 1), A(\gt 1)), (A(\gs 2), A(\gt1)),  (A(\gs), A(\gt 2))}$, or \\ $\set{(A(\gs), A(\gt 1)), (A(\gs 1), A(\gt2)),  (A(\gs2), A(\gt 2))}$, i.e., the set corresponding to each $(\gs, \gt) \in \set{1, 2}^{\ell(n)\ast 2} \setminus I$ can be chosen in two different ways, and the set corresponding to each $(\gs, \gt) \in I$ can be chosen in four different ways.
Since $\te{card}(\set{1, 2}^{\ell(n)\ast 2} \setminus I)=4^{\ell(n)}-(n-2\cdot4^{\ell(n)})=3 \cdot 4^{\ell(n)}-n$ and the subset $I$ can be chosen from $\set{1, 2}^{\ell(n)\ast 2}$ in ${}^{4^{\ell(n)}}C_{n-2 \cdot 4^{\ell(n)}}$ ways, the number of the sets $\ga_n(I)$ is $2^{3 \cdot 4^{\ell(n)}-n}\cdot 4^{\te{card}(I)}\cdot{}^{4^{\ell(n)}}C_{n-2 \cdot 4^{\ell(n)}}$.
\end{remark}

We now give an example illustrating Definition~\ref{defi3}.
\begin{example} Let $n=9$. Then, $\ell(n)=1$, $I\sci \set{1, 2}^{\ast 2}$ with $\te{card}(I)=1$. Take $I=\set{(1, 1)}$. Then,
\begin{align*}
\ga_9(\set{(1,1)})&=\set{(A(11), A(2)), (A(12), A(2)), (A(21), A(2)), (A(22), A(2)), (A(21),A(1)), \\
& (A(22), A(1))}\uu\set{(A(11), A(1)), (A(12), A(11)), (A(12), A(12))}\\
&=\{(\frac 1{18}, \frac 56), (\frac {5}{18}, \frac 56), (\frac {13}{18}, \frac {5}{6}), (\frac {17}{18}, \frac {5}{6}), (\frac {13}{18}, \frac {1}{6}), (\frac {17}{18}, \frac {1}{6})\}\\
&\qquad \uu \set{(\frac 1{18}, \frac 1 6), (\frac 5{18}, \frac 1{18}), (\frac 5{18}, \frac 5{18})}.
\end{align*}
Note that each of $\ga_9(\set{(1,1)})$, $\ga_9(\set{(1,2)})$, $\ga_9(\set{(2,1)})$, $\ga_9(\set{(2,2)})$ can be chosen in $32$ ways, i.e., the numbers of the sets $\ga_9(I)$ in this case is $4 \cdot 32=128$. Moreover, using the formula in Remark~\ref{rem10}, we have
\[2^{3 \cdot 4^{\ell(n)}-n}\cdot 4^{\te{card}(I)}\cdot{}^{4^{\ell(n)}}C_{n-2 \cdot 4^{\ell(n)}}=128.\]
\end{example}

\begin{prop} \label{prop5}
{\it Let $n\geq 4$ and $\ga_n(I)$ be the set as defined in Definition~\ref{defi3}. Then, $\ga_n(I)$ forms an optimal set of $n$-means with quantization error
\[V_n= \frac 1{36^{\ell(n)+1}} (9\cdot 4^{\ell(n)}-2n).\]}
\end{prop}
\begin{proof} We have $n=2 \cdot 4^{\ell(n)}+k$ where $1\leq k<2 \cdot 4^{\ell(n)}$. Set $\gb_{ij}=\ga\ii A_{ij}$ with $n_{ij}=\te{card}(\gb_{ij})$ for $1\leq i, j\leq 2$. Let us prove it by induction. We first assume $k=1$. By Lemma~\ref{lemma41} and Lemma~\ref{lemma22}, we can assume that each of $U_{(i, j)}^{-1}(\gb_{ij})$ for $i=2$ and $j=1, 2$, are optimal sets of $2 \cdot 4^{\ell(n)-1}$-means and $U_{(1,1)}^{-1}(\gb_{11})$ is an optimal set of $(2\cdot 4^{\ell(n)-1}+1)$-means. Thus, for $i=2$ and $j=1, 2$, we can write
\begin{align*} U_{(i, j)}^{-1}(\gb_{ij})&=\set{U_{(\gs, \gt)}(\ga_2) : (\gs, \gt) \in \set{1,2}^{(\ell(n)-1)\ast 2}}, \te{ and }\\
U_{(1, 1)}^{-1}(\gb_{11})&=\set{U_{(\gs, \gt)}(\ga_2) : (\gs, \gt) \in \set{1,2}^{(\ell(n)-1)\ast 2}\setminus \set{\gt}}\uu U_\gt(\ga_3),
\end{align*}
for some $\gt \in \set{1,2}^{(\ell(n)-1)\ast 2}$, where $\ga_3$ is an optimal set of three-means. Thus
\[\ga_n(\set{(1,1)\gt})=\mathop{\uu}\limits_{i, j=1}^2 \gb_{ij}=\set{U_{(\gs, \gt)}(\ga_2) : (\gs, \gt) \in \set{1,2}^{\ell(n)\ast 2}\setminus \set{(1,1)\gt}}\uu U_{(1,1)\gt}(\ga_3),\]
for some $\gt \in \set{1,2}^{(\ell(n)-1)\ast 2}$, where $\ga_3$ is an optimal set of three-means. Notice that instead of choosing $U_{(1,1)}^{-1}(\gb_{11})$ as an optimal set of $(2\cdot 4^{\ell(n)-1}+1)$-means, one can choose any one from $U_{(i, j)}^{-1}(\gb_{ij})$ for $i=2$, $j=1, 2$, as an optimal set of $(2\cdot 4^{\ell(n)-1}+1)$-means. Hence, for $n=2\cdot 4^{\ell(n)}+1$, one can write
\[\ga_n(I)=\mathop{\uu}\limits_{i, j=1}^2 \gb_{ij}=\set{U_{(\gs, \gt)}(\ga_2) : (\gs, \gt) \in \set{1,2}^{\ell(n)\ast 2}\setminus \set{\gt}}\uu U_{\gt}(\ga_3),\]
where $I=\set{\gt}$ for some $\gt \in \set{1,2}^{\ell(n)\ast 2}$ as an optimal set of $n$-means. Thus, we see that the proposition is true if $n=2\cdot 4^{\ell(n)}+1$. Similarly, one can prove that the proposition is true for any $1\leq k<2\cdot 4^{\ell(n)}$. Thus, writing $\ga_2=\set{(A(1), A(\es)), (A(2), A(\es))}$, and $\ga_3=\set{(A(1), A(1)), (A(1), A(2)), (A(2), A(\es))}$, we have, in general,
\begin{align*} \ga_n(I)&=\mathop{\uu}\limits_{(\gs, \gt) \in \set{1, 2}^{\ell(n)\ast 2} \setminus I}\set{(A(\gs 1), A(\gt)), (A(\gs 2), A(\gt))}\\
 &\uu (\mathop{\uu}\limits_{ (\gs, \gt) \in I} \set{(A(\gs 1), A(\gt 1)), (A(\gs 1), A(\gt2)),  (A(\gs 2), A(\gt))}),
\end{align*}
where $I\sci \set{1, 2}^{\ell(n)\ast 2}$ with $\te{card}(I)=k$ for some $1\leq k<2 \cdot 4^{\ell(n)}$. Then, we obtain the quantization error as
\begin{align*}
&V_n=\min_{(a, b)  \in \gb_n(I)}\|x-(a, b)\|^2 dP= \sum_{(\gs, \gt) \in \set{1, 2}^{\ell(n)\ast 2} \setminus I}\sum_{i=1}^2 \mathop\int\limits_{A_{\gs i}\times A_\gt}\|x-(A(\gs i), A(\gt))\|^2 d(P_c\times P_c) \\
&\qquad \qquad  +\sum_{(\gs, \gt) \in  I}\Big(\sum_{j=1}^2 \mathop\int\limits_{A_{\gs 1}\times A_{\gt j}}\|x-(A(\gs 1), A(\gt j))\|^2 d(P_c\times P_c)\\
&\qquad \qquad + \mathop\int \limits_{A_{\gs 2}\times A_{\gt}}\|x-(A(\gs 2), A(\gt))\|^2 d(P_c\times P_c)\Big) \\
&=\sum_{(\gs, \gt) \in \set{1, 2}^{\ell(n)\ast 2} \setminus I}\sum_{i=1}^2  \frac 1{4^{\ell(n)}} \frac 12 (u_{\gs i}^2+u_{\gt}^2)\frac 1 8+\sum_{(\gs, \gt) \in  I}\frac 1{4^{\ell(n)}} \Big(\sum_{j=1}^2  \frac 14 (u_{\gs 1}^2+u_{\gt j}^2)\frac 1 8+ \frac 12(u_{\gs 2}^2+u_{\gt}^2)\frac 1 8\Big)\\
&=\sum_{(\gs, \gt) \in \set{1, 2}^{\ell(n)\ast 2} \setminus I} \frac 1{4^{\ell(n)}} (\frac 1 9u_{\gs}^2+u_{\gt}^2)\frac 1 8+\sum_{(\gs, \gt) \in  I} \frac 1{4^{\ell(n)}} (u_{\gs}^2+5 u_{\gt}^2)\frac 1 {72}.
\end{align*}
Since, $\te{card}(\set{1, 2}^{\ell(n)\ast 2} \setminus I)=3\cdot 4^{\ell(n)}-n$, $\te{card}(I)=n-2\cdot 4^{\ell(n)}$, $u_\gs=u_\gt=\frac 1{3^{\ell(n)}}$, upon simplification, we have $V_n= \frac 1{36^{\ell(n)+1}} (9\cdot 4^{\ell(n)}-2n)$.
Thus, the proof of the proposition is complete.
\end{proof}

%\begin{remark} It is well-known that the optimal set of one-mean for a probability distribution is always the expected value of the distribution and the corresponding quantization error is the variance. Propositions \ref{prop1}, \ref{prop2}, \ref{prop31}, \ref{prop4} and \ref{prop5} give all the optimal sets of $n$-means for all $n\geq 2$, and the corresponding quantization error for the affine measure $P$.

%\end{remark}

\section{Quantization dimension and quantization coefficient for $P$}

The techniques employed in the previous sections also provide closed formulas for the quantization errors involved at each step.  Such closed formulas are amenable for direct calculation of the quantization dimension and the quantization coefficient for the probability distribution involved.  Hence, in this section we will calculate the quantization dimension $D(P)$ of the probability distribution $P$, and the accumulation points for the $D(P)$-dimensional quantization coefficients. By Proposition~\ref{prop31}, Proposition \ref{prop4}, and Proposition \ref{prop5},
the $n$th quantization error $V_n$ is given by
\begin{equation} \label{eq200} V_n=\left\{\begin{array} {ll}
\frac 14 \frac 1{36^{\ell(n)}} \Big(2 \cdot 4^{\ell(n)}-n +\frac 5 9(n-4^{\ell(n)})\Big) & \te{ if }4^{\ell(n)} \leq n\leq 2\cdot 4^{\ell(n)},\\
\frac 1{36^{\ell(n)+1}} (9\cdot 4^{\ell(n)}-2n) &\te{ if } 2\cdot 4^{\ell(n)} < n<  4^{\ell(n)+1}.
\end{array}
\right.
\end{equation}

\begin{prop} \label{prop6}
{\it The quantization dimension $D(P)$ of the probability distribution $P$ exists and equals $\frac{\log 4}{\log 3}. $}
\end{prop}

\begin{proof}
By \eqref{eq200}, for $4^{\ell(n)} \leq  n\leq 2\cdot 4^{\ell(n)},$ it follows that $V_{2\cdot 4^{\ell(n)}}\leq V_n\leq V_{4^{\ell(n)} }$, i.e.,
\[\frac 5{36} 9^{-\ell(n)}\leq V_n\leq \frac 1 4 9^{-\ell(n)},\] and so
\[\frac{2\ell(n) \log 4}{-\log \frac 5{36} +\ell(n) \log  9}\leq  \frac {2\log n}{-\log V_n}<\frac{2\log 2+2\ell(n)\log 4}{-\log \frac 1 4 +\ell(n) \log 9}.\]
Thus, we deduce that
\[ \lim_{n\to \infty} \frac{2 \log n}{-\log V_n} =\frac{\log 4}{\log 3}. \]
Similarly, for $2\cdot 4^{\ell(n)} < n<  4^{\ell(n)+1}$, we also obtain the same limit.  Hence,
$$D(P)=\lim_{n\to \infty} \frac{2 \log n}{-\log V_n}=\frac{\log 4}{\log 3}.$$
Thus, the proof of the proposition is complete.
\end{proof}

\begin{prop} \label{prop7}
Let $\gb:=D(P)$ be the quantization dimension of $P$. Then, the $\gb$-dimensional quantization coefficient for$P$ does not exist, and the accumulation points of $\{n^{\frac{2}{\beta}} V_n \}_{n\in \D N}$ lie in the closed interval $[\frac{1}{12}, \frac{5}{4}]$.
\end{prop}
\begin{proof}
Recall the sequence of quantization errors $\set{V_n}_{n=4}^\infty$ given by \eqref{eq200}. Again, notice that $4^{\frac 1{\gb}}=3$.
Along the sequence $\{4^{\ell(n)}\}_{n\in \D N}$, we have
$\lim_{n\to \infty}  (4^{\ell(n)})^{\frac{2}{\beta}} V_{4^{\ell(n)}}=\frac{1}{4}.$  Similarly, along the sequence
$\{2\cdot 4^{\ell(n)}\}_{n\in \D N}, $ we have $\lim_{n\to \infty} (2\cdot4^{\ell(n)})^{\frac{2}{\beta}} V_{2. 4^{\ell(n)}} =\frac{5}{12}. $  Consequently, $\lim\limits_{n \to \infty} n^{\frac{2}{\beta}} V_n $ does not exist. Now, we calculate the range for the accumulation points of $\{n^{\frac{2}{\beta}} V_n \}_{n\in \D N}$. The following two cases can arise:

\tit{Case~1.  $4^{\ell(n)}\leq n\leq 2\cdot 4^{\ell(n)}$. }

In this case, we have
$V_{2. 4^{\ell(n)}} \leq V_n \leq V_{4^{\ell(n)}},$
implying
$ (4^{\ell(n)})^{\frac{2}{\beta}} V_{2\cdot4^{\ell(n)}} \leq n^{\frac{2}{\beta}} V_n \leq (2\cdot 4^{\ell(n)})^{\frac{2}{\beta}} V_{4^{\ell(n)}}. $
Since
\[\lim_{n\to \infty} (4^{\ell(n)})^{\frac{2}{\beta}} V_{2\cdot4^{\ell(n)}}=\frac {5}{36}, \te{ and } \lim_{n\to \infty} (2\cdot 4^{\ell(n)})^{\frac{2}{\beta}} V_{4^{\ell(n)}}=\frac 34,\]
it follows that along such subsequences, we have
$\liminf_n n^{\frac{2}{\beta}} V_n =\frac{5}{36} < \frac{3}{4} =
\limsup_n n^{\frac{2}{\beta}} V_n.$

\tit{Case~2.  $2\cdot 4^{\ell(n)}<n<4^{\ell(n)+1} $.}

In this case, we have
$V_{4^{\ell(n)+1}} < V_n<V_{2\cdot 4^{\ell(n)}}$,
implying
$$ (2\cdot 4^{\ell(n)})^{\frac{2}{\beta}} V_{4^{\ell(n)+1}}<n^{\frac{2}{\beta}} V_n <(4^{\ell(n)+1})^{\frac{2}{\beta}} V_{2\cdot 4^{\ell(n)}}. $$
Since
\[\lim_{n\to \infty} (2\cdot 4^{\ell(n)})^{\frac{2}{\beta}} V_{4^{\ell(n)+1}}=\frac {1}{12}, \te{ and } \lim_{n\to \infty}(4^{\ell(n)+1})^{\frac{2}{\beta}} V_{2\cdot 4^{\ell(n)}}=\frac 54,\]
it follows that
$\liminf_n n^{\frac{2}{\beta}} V_n =\frac {1}{12} < \frac 54=
\limsup_n n^{\frac{2}{\beta}} V_n.$

By Case~1 and Case~2, for $n\in \D N$, we see that
\[\liminf_n n^{\frac{2}{\beta}} V_n =\frac {1}{12} < \frac 54=
\limsup_n n^{\frac{2}{\beta}} V_n,\]
which yields the fact that the accumulation points of $\{n^{\frac{2}{\beta}} V_n \}_{n\in \D N}$ lie in the closed interval $[\frac{1}{12}, \frac{5}{4}]$.
Thus, the proof of the proposition is complete.
\end{proof}

\section{Discussion and Concluding Remarks}

{\bf Motivation.}  As it has been mentioned in Introduction, the main motivation for this article is completion of the programme initiated in \cite{MR}. In the meantime, we extend the results in \cite{CR} to the setting of infinite affine transformations.  Analogously to \cite{GL3}, this completes the programme of providing complete quantization for affine measures on $\mathbb{R}^2. $
\medskip

{\bf Observations and Remarks.}  Quantization of continuous random signals (or random variables and processes) is an
important part of digital representation of analog signals for various coding techniques
(e.g., source coding, data compression, archiving, restoration). The oldest example of quantization in statistics is rounding off. Sheppard (see \cite{S1}) was the first who analyzed rounding off for estimating densities by histograms.
Any real number $x$ can be rounded off (or quantized) to the nearest
integer, say $q(x) = [x]$, with a resulting quantization error $e(x) = x - q(x) . $
%, for example, $q(2.14259) =2$.
Hence, the restored signal may differ
from the original one and some information can be lost.
 Thus, in quantization of a continuous set of values there is always a distortion (also known as noise or error) between the original set of values and the quantized set of values. The main goal in quantization theory is finding a set of quantizers with minimum distortion, which has been extensively investigated by numerous authors \cite{CG, GL, GN, LCG, SS1, Z}.
%For the most comprehensive overview of quantization one can see \cite{GN} (for later references, see \cite{GL}).   Over the years several authors estimated the distortion measures for quantizers (see, e.g., \cite{LCG} and \cite{Z}). A class of asymptotically optimal quantizers with respect to an $r$th-mean error distortion measure is considered in \cite{GL2} (see also \cite{CG, SS1}).
A different approach for uniform scalar quantization is
developed in \cite{SS2}, where the correlation properties of a Gaussian process are exploited to
evaluate the asymptotic behavior of the random quantization rate for uniform quantizers.
General quantization problems for Gaussian processes in infinite-dimensional functional
spaces are considered in \cite{LP}. In estimating weighted integrals of time series with no quadratic mean derivatives, by means of samples at discrete times, it is known that the rate of convergence of mean-square error is reduced from $n^{-2}$ to $n^{-1.5}$ when the samples are quantized (see \cite{BC1}). For smoother time series, with $k=1,2, \cdots$ quadratic mean derivatives, the rate of convergence is reduced from $n^{-2k-2}$ to $n^{-2}$ when the samples are quantized, which is a very significant reduction (see \cite{BC2}). The interplay between sampling and quantization is also studied in \cite{BC2}, which asymptotically leads to optimal allocation between the number of samples and the number of levels of quantization. Quantization also seems to be a promising tool in recent development in numerical probability (see, e.g., \cite{PPP}).

By Proposition~\ref{prop000} the points in an optimal set are the centroids of their own Voronoi regions. Consequently, the points in an optimal set are an evenly-spaced distribution of sites in the domain with minimum distortion error with respect to a given probability measure and is therefore very useful in many fields, such as clustering, data compression, optimal mesh generation, cellular biology, optimal quadrature, coverage control and geographical optimization, for more details one can see \cite{DFG, OBSC}. Besides, it has  applications in energy efficient distribution of base stations in a cellular network  \cite{HCHSVH, KKR, S2}. In both geographical and cellular applications
the distribution of users is highly complex and often modeled by a fractal \cite{ABDHW, LZSC}.

\medskip

{\bf Future Directions.} $k$-means clustering is a method of vector quantization, originally from signal processing, that aims to partition $n$ observations, or the underlying data set into $k$ clusters in which each observation belongs to the cluster with the nearest mean, also known as cluster center or cluster centroid. For a given $k$ and a given probability distribution in a data set there can be two or more different sets of $k$-means clusters: for example, with respect to a uniform distribution the unit square $\set{(x_1, x_2) : |x_1|\leq 1, |x_2|\leq 1}$ has four different sets of two-means clusters with cluster centers $\set{(\frac 12, \frac 12), \, (-\frac 12, -\frac 12)}$, $\set{(-\frac 12, \frac 12), \, (\frac 12, -\frac 12)}$, $\set{(-\frac 12, 0), \, (\frac 12, 0)}$, and $\set{(0, \frac 12), \, (0, -\frac 12)}$. Among these only $\set{(-\frac 12, 0), \, (\frac 12, 0)}$, and $\set{(0, \frac 12), \, (0, -\frac 12)}$ form two different optimal sets of two-means. In other words, we can say that for a given $k$, among the multiple sets of $k$-means clusters, the centers of a set with the smallest distortion error form an optimal set of $k$-means. Thus, it is much more difficult to calculate an optimal set of $k$-means than to calculate a set of $k$-means clusters. There are several work done in the direction of $k$-means clustering. On the other hand, there is not much work in the direction of finding optimal sets of $k$-means clusters, and the work in this paper is an addition in this direction.

The probability measure $P$ considered in this study has identical marginal distributions, which is instrumental in determining optimal sets of 2-, 3-, and 4-means accurately.  Besides, it enables us to bridge infinitely generated affine measures with finitely generated ones, and consequently, connect optimal sets of $n$-means for $P$ and $P_C \times P_C . $  It would be interesting to investigate if similar results can be achieved when $P$ is induced by different infinite probability vectors $\{p_{ij}\}$ than considered in this article.

\end{document}